\documentclass[12pt,reqno]{amsart}
\usepackage{amsmath}
\usepackage{txfonts}
\usepackage{amsfonts}
\usepackage{a4wide}
\usepackage{amsmath,amsthm,amssymb,amscd}
\usepackage{latexsym}
\usepackage{hyperref}
\usepackage[numbers,sort&compress]{natbib}
\usepackage{hypernat}
\usepackage{color,enumitem,graphicx}
\usepackage{appendix}
\allowdisplaybreaks
\numberwithin{equation}{section}

\newtheorem{theorem}{Theorem}[section]

\newtheorem{corollary}[theorem]{Corollary}
\newtheorem{lemma}[theorem]{Lemma}

\theoremstyle{definition}

\begin{document}
	\title [Planar Schr\"{o}dinger-Poisson system with logarithmic nonlinearity]{Multiple solutions for planar Schr\"{o}dinger-Poisson system with logarithmic nonlinearity}

	\author{Wei Long \textsuperscript{1} ,\,\, Changchang Yan \textsuperscript{1}$^{, \ast}$, Jianghua Ye\textsuperscript{2}}
	
	\thanks{\textsuperscript{1}  School of Mathematics and Statistics, and Jiangxi Provincial Center for Applied Mathematics, Jiangxi Normal University, Nanchang, Jiangxi 330022, People's Republic of China}	
	\thanks{\textsuperscript{2}  School of Mathematics, and Center for Applied Mathematics of Guangxi (Guangxi University), Guangxi University, Nanning, Guangxi 530004, People's Republic of China}

	\thanks{Email:  lwhope@jxnu.edu.cn}
	\thanks{Email: mathycc@163.com}
    \thanks{Email: jhye@gxu.edu.cn}
	\thanks{*Corresponding author}

	\begin{abstract}		
This paper is concerned with the following  planar Schr\"{o}dinger-Poisson system 
\begin{align}\notag
    \
		\begin{cases} 
        -\Delta u+V(x)u+\lambda\phi u=u\ln u^2, &\quad \mathrm{in}\quad \mathbb{R}^2, \\ 
        \Delta \phi=u^{2}, &\quad \mathrm{in} \quad \mathbb{R}^2,\\
		\end{cases} \
\end{align}
where $\lambda\in \mathbb{R}$ is a parameter and $ V\in C^1(\mathbb{R}^2,\mathbb{R^+})$ is a coercive potential. Due to the presence of the logarithmic nonlinearity, the Nehari-Pohozaev manifold method developed by Ruiz \cite{R} is not applicable here. We apply a general minimax principle to an auxiliary functional and prove the existence of a nontrivial solution. Then, by minimizing the energy functional over the set of all nontrivial solutions, we succeed in showing that the system admits a ground state solution. If, in addition, $V$ is radially symmetric, we obtain infinitely many nonradial sign-changing solutions.
%First, we prove the existence of a nontrivial solution   by applying a general minimax principle to an auxiliary functional. Then, by minimizing the energy functional over the set of all nontrivial solutions, we succeed in showing that the system admits a ground state solution. If, in addition, $V$ is radially symmetric, we obtain the existence of infinitely many nonradial sign-changing solutions. However, the scaling parameter $t$ in the integral $\int_{\mathbb{R}^2}V(tx)u^2(x)\,dx$ cannot be directly extracted, and $u\ln u^2$ is sublinear near $0$, which necessitates a series of precise and careful analyses.
%By the constrained minimization method, we prove the  existence of ground state solution and sign-changing solution   with general potential and logarithmic nonlinear term. Furthermore, certain restrictions on  potential $V$ are required to establish the existence of ground state solutions.

\vspace{.1cm}
	\end{abstract}
		
	\maketitle
	\section{Introduction}
	In this paper, we  consider the following Schr\"{o}dinger-Poisson   system 
	\begin{align}\label{lem1-1}
		\
		\begin{cases} i \psi_{t}-\Delta \psi+Q(x) \psi+\lambda\varphi\psi=f(x,\psi),&\quad \mathrm{in}\ \ \mathbb{R}^{N} \times \mathbb{R}, \\ \Delta \varphi=|\psi|^{2}, &\quad \mathrm{in}\  \ \mathbb{R}^{N} \times\mathbb{R},\\
		\end{cases} \
	\end{align}
	where the function  \(\psi: \mathbb{R}^{N} \times\mathbb{R} \to \mathbb{C}\ (N \geq2)\)  {is the} time-dependent wave function, the function $ \varphi$ is an internal potential for a nonlocal self-interaction of the wave function $\psi$, the function \(Q\) is a real external potential, the nonlinearity $f(x,\psi)$ models the interaction among particles and $\lambda$ is a parameter. 
    System \eqref{lem1-1} is of considerable physical significance; it arises, for instance, in quantum mechanics (see \cite{BBL,CL,L}), semiconductor theory (see \cite{BF,MRS}), and plasma physics (see \cite{Lions}). In particular, systems like \eqref{lem1-1} have been introduced in \cite{BF} as a model describing solitary waves, for nonlinear stationary equations of Schr\"{o}dinger type interacting with an electrostatic field.
    
   % To solve the Poisson equation  $\Delta \varphi=|\psi|^{2}$, we introduce the fundamental solution of the Laplacian operator, denoted by $\Phi_N(x)$, which satisfies
	
	By adopting the standing wave ansatz $\psi(x,t)=e^{-i\mu t}u(x)$ for $\mu\in \mathbb{R}$,	the Schr\"{o}dinger-Poisson system \eqref{lem1-1} can be converted to 		
	\begin{align}\label{lem1-2}
		-\Delta u+V(x)u+\lambda[\Phi_N * u^2]u=f(x,u),\quad \mathrm{in} \quad  \mathbb{R}^N,
	\end{align}
   % cu\ln u^2 
   where $V(x)=Q(x)+\mu$, $\Phi_N$ is the fundamental solution of the Laplacian operator given by 
   \[
	\Phi_N(x) :=
	\begin{cases}
		\dfrac{1}{N(2-N)\omega_N}|x|^{2-N}, & N \geq 3, \\
		\dfrac{1}{2\pi}\ln|x|, & N = 2,
	\end{cases}
	\]
	and	\(\omega_{N}>0\) denotes the volume of the unit ball in \(\mathbb{R}^{N}\).
   %The equations \eqref{lem1-1}  and \eqref{lem1-2}  possess extensive physical applications in diverse physical fields: quantum mechanics, quantum optics, nuclear physics, transport-diffusion phenomena, open quantum systems, effective quantum gravity and Bose–Einstein condensation; see  \cite{6,7,8,9} and the references therein.

   In the three-dimensional case, the Schr\"{o}dinger-Poisson system with power nonlinearity 
   \begin{equation}\label{seqs12}
 \left\{
 \renewcommand{\arraystretch}{1.25}
 \begin{array}{ll}
 -\Delta u+V(x)u+\lambda \phi u = |u|^{p-2}u, \ & \mathrm{in}\  \ \mathbb{R}^3,\\
-\Delta \phi=u^2, \ & \mathrm{in}\  \ \mathbb{R}^3
 \end{array}
 \right.
 \end{equation}
 has been widely investigated, especially regarding the existence of positive solutions, sign-changing solutions, ground state solutions, and radial solutions; see \cite{A, BJL, GPV, WZ, JZ, LM1,JWZ,R,SSY1,CV,LLTY} and the references therein.
 When $V\equiv1$, Ruiz \cite{R} proved that system \eqref{seqs12} admits a positive radial solution for $p\in(3,6)$ and $\lambda>0$ by using the Nehari-Pohozaev manifold method. Moreover, for $p\in(2,3]$, he showed that the existence of solutions to system \eqref{seqs12} depends on the parameter $\lambda$. Specifically, system \eqref{seqs12} has no nontrivial solutions for $\lambda\in [\frac{1}{4},\infty)$, but admits at least one positive radial solution for $\lambda>0$ sufficiently small.
 When $V(x)$ is a coercive potential, i.e., $\lim\limits_{|x|\rightarrow \infty}V(x)=+\infty$, system \eqref{seqs12} was studied in \cite{WZ,JWZ}. Using a constraint variational method combined with Brouwer degree theory, Wang and Zhou \cite{WZ} established the existence of a sign-changing solution to system \eqref{seqs12} with $4<p<6$. In \cite{JWZ}, Jiang et al. treated the case $p\in (2,3)$ and proved that system \eqref{seqs12} possesses at least two positive solutions for $\lambda>0$ small enough.

%Much attention has been paid to the problem \eqref{lem1-2} with the nonlinear term $f(x,u)=|u|^{p-2}u$,  case $N \geq 3$ of the problem \eqref{lem1-2} has received extensive research over the last two decades;   

\vspace{.1cm}
However, in the two-dimensional case, $\Phi_2(x)$ is sign-changing and singular as $|x|$ approaches zero and infinity, which leads to substantial differences from the case $N=3$.
When $N=2$ and $f(x,u)=|u|^{p-2}u$, equation \eqref{lem1-2} becomes 
\begin{align}\label{1.4}
    -\Delta u+V(x)u+\lambda[\Phi_2 * u^2]u=|u|^{p-2}u,\quad \mathrm{in} \quad  \mathbb{R}^2.
\end{align}
The energy functional corresponding to equation \eqref{1.4} is 
\begin{align*}%\label{1-5}
		J(u)=\dfrac{1}{2}\int_{\mathbb{R}^2}\left[ |\nabla u|^2+V(x)u^2\right]dx +\frac{\lambda}{4}\int_{\mathbb{R}^2}\int_{\mathbb{R}^2} \ln |x-y|u^2(x)u^2(y)dxdy-\frac{1}{p}\int_{\mathbb{R}^2}|u|^pdx.
	\end{align*}
However, the functional $J$ is not well-defined on $H^1(\mathbb{R}^2)$. To
overcome this difficulty, inspired by \cite{S1}, Cingolani and Weth \cite{CW} investigated equation \eqref{1.4} in a smaller Hilbert space
\begin{align}\label{X}
    X := \left\{ u \in H^1(\mathbb{R}^2) \,\bigg|\, \int_{\mathbb{R}^2} \ln(1+|x|) u^2 dx < \infty \right\}
\end{align}
equipped with the norm 
\begin{align*}
  \|u\|_X:=\left[\int_{\mathbb{R}^2} \left(|\nabla u|^2+V(x)u^2+\ln(1+|x|) u^2\right) dx\right]^{\frac{1}{2}}. 
\end{align*}
Then, $J\in C^1(X,\mathbb{R})$. When $p\geq 4$ and $V\in L^{\infty}(\mathbb{R}^2)$ is continuous and $\mathbb{Z}^2$-periodic, Cingolani and Weth proved that equation \eqref{1.4} possesses infinitely many geometrically distinct solutions, as well as ground state solutions. If, in addition, $V$ is a positive constant, they showed the existence of nonradial solutions to equation \eqref{1.4} with arbitrarily many nodal domains. Later,  Du and Weth \cite{DW} treated the case $p>2$ and $V\equiv1$, and obtained that equation \eqref{1.4} admits a ground state solution and infinitely many nonradial sign-changing solutions.
For more results on equation \eqref{1.4}, we refer the reader to \cite{Az,CRT,GLL,LRZ,LRTZ,FW} and the references therein.
%\begin{align}\label{lem1-3}
%-\Delta u+V(x)u+m [\Phi_2 * u^2]u=|u|^{p-2}u,\quad \mathrm{in}\ \mathbb{R}^2,
%\end{align}
%where $2<p<4$ and $V(x)\in L^{\infty}(\mathbb{R}^2)$.  They proved the existence of a sign-changing solution of \eqref{lem1-3} by suitably defined actions of subgroups of the orthogonal group $O(2)$.

\vspace{.1cm}
When $\lambda=0$  and \(f(x,u)=u\ln u^2\), equation  \eqref{lem1-2} becomes 
\begin{align}\label{1.5}
-\Delta u+V(x)u=u\ln u^2, \quad \mathrm{in} \ \ \mathbb{R}^N,
\end{align}
which has been studied in \cite{CG,DMS,S,T,WZ1}.
If $V$ is a constant and $N\geq 3$, d’Avenia et al. \cite{DMS} proved that equation \eqref{1.5} admits a unique positive solution, which is radially symmetric and nondegenerate using the nonsmooth critical point theory. In \cite{T}, Troy dealt with the case $V\equiv1$ and $N\in [1,9]$, and showed that the positive radial solution of equation \eqref{1.5} is unique.
By utilizing direction derivative and the Nehari-type constrained minimization method, Shuai proved the existence of positive and sign-changing solutions to equation \eqref{1.5} under different assumptions on $V$ in \cite{S}.

%many researchers have made efforts to study equation  \eqref{lem1-2}. For example, Squassina and Szulkin \cite{SS2,SS1} studied the following logarithmic Schr\"{o}dinger equation
%where $V(x)$ and $Q(x)$ are spatially periodic. They showed the existence of positive ground-state solution and infinitely many high energy solutions by  non-smooth critical point theory for lower semi-continuous functionals.
% for problem \eqref{lem1-2} when $V(x)$ is constant. If $ V(x)$ is radially symmetric, Shuai \cite{S} showed the existence of  sign-changing solutions for equation \eqref{lem1-2} using the constrained minimization method. Furthermore,  Azzollini et al. \cite{ADP} investigate the following equation 
%\begin{align}\label{lem1-4}
%-\Delta u+(\ln|\cdot|\ast |u|^2)u=(\ln |\cdot|\ast |u|^q)|u|^{q-2}u,\quad \mathrm{in}\ \mathbb{R}^2,
%\end{align}
%where $\frac{8}{3}<q<4$. Using variational arguments, they proved that there exist infinitely many radially symmetric classical solutions for equation \eqref{lem1-4}. 
\vspace{.1cm}
However, there are few results on the Schr\"{o}dinger-Poisson system with logarithmic nonlinearity. 
%Another hot topic is the existence of solutions for equation \eqref{lem1-1} with \(\lambda\not \neq 0\) and logarithmic nonlinearity $u\ln u^2$.
In dimension three, Wang et al. \cite{WZZ} studied the following equation
\begin{align}\label{p1}
   \begin{cases}
-\Delta u + V(x)u + \lambda \phi(x)u = u\ln u^2, \ \ \mathrm{in}\  \ \mathbb{R}^3,\\
-\Delta \phi = u^2,\quad \displaystyle\lim_{|x|\to+\infty}\phi(x)=0,
\end{cases} 
\end{align}
where $V\in C(\mathbb{R}^3, \mathbb{R}^+)$ is a coercive potential. They proved that  there exists $\lambda_0 > 0$ such that equation \eqref{p1} has a ground state solution for any $\lambda \in (0,\lambda_0)$. %, which blows up as $\lambda \to 0^+$ if $V(x) \sim (\log|x|)^{\frac12}$ at infinity.
Moreover, a nontrivial solution to equation \eqref{p1} with $\lambda<0$ is obtained.
Concerning dimension two, Dolbeault et al. \cite{DFJ} investigated the following equation
\begin{align}\label{1-4}
		-\Delta u+V(x)u+\lambda\left( \int_{\mathbb{R}^2}\ln |x-y| u^2(y)dy\right) u=u\ln u^2, \quad \mathrm{in}\quad \mathbb{R}^2
	\end{align}
    with the special potential $V(x)=2\ln (1+|x|^2)$. They studied the ground state energy of equation \eqref{1-4} by using some new logarithmic interpolation inequalities.

    \vspace{.1cm}
Motivated by the above results, in this paper, we investigate the equation \eqref{1-4} with a general potential. Specifically, we assume that $V$ satisfies the following conditions:

\vspace{.1cm}
\noindent $(V_1) $ $ V\in C^1(\mathbb{R}^2,   \mathbb{R^+})$ and there exist $\gamma>0$ and $m\in (0,+\infty]$ such that $\mathop {\lim }\limits_{|x| \to \infty } \frac{V(x)}{|x|^\gamma}=m$;
	%\\$ (V_2) $  $\left(\nabla V\cdot x \right)\geq 0$ for any $ x\in \mathbb{R}^2. $
	\\$ (V_2)$ there exists $\kappa>0$ such that $|\nabla V(x)\cdot x|\leq \kappa V(x)$ for any $x\in \mathbb{R}^2.$

\vspace{.1cm}
	Formally, the energy functional corresponding to equation \eqref{1-4} is 
    \begin{align*}%\label{1-5}
			I(u)\notag&=\dfrac{1}{2}\int_{\mathbb{R}^2}\left[ |\nabla u|^2+u^2+V(x)u^2\right]dx -\frac{1}{2}\int_{\mathbb{R}^2}u^2\ln u^2dx\\&\quad +\frac{\lambda}{4}\int_{\mathbb{R}^2}\int_{\mathbb{R}^2} \ln |x-y|u^2(x)u^2(y)dxdy.
        \end{align*}
       However, due to the appearance of the logarithmic nonlinearity, the functional $I$ is not well-defined on $X$ defined by \eqref{X}. In fact, the function $\varphi$ given in \cite[Example 2.1]{WZZ} satisfies $\varphi\in X$ and $\int_{\mathbb{R}^2}\varphi^2\ln \varphi^2dx=-\infty$.
       In order to overcome this difficulty, we introduce the space 
        $$\textit{E}:=\left\lbrace u\in H^1{(\mathbb{R}^2)\ |\ \int_{\mathbb{R}^2} V(x)u^2dx<\infty } \right\rbrace,$$
        which is a Hilbert space equipped with the following inner product and norm:
        \begin{align*}
            (u,v)_E:=\int_{\mathbb{R}^2}\left[ \nabla u\nabla v+V(x)uv\right] dx,\ \ 
            ||u||_E:=(u,u)_E^{\frac{1}{2}}.
        \end{align*}
  In fact, by the definition of $V$, we obtain $E\subset X$.      %$$||u||_E^2=\int_{\mathbb{R}^2}\left[ |\nabla u|^2+V(x)u^2\right] dx.$$
		%Let 
		%$$H^1(\mathbb{R}^2):=\left\lbrace u\in L^2{(\mathbb{R}^2)\ |\ \nabla u \in L^2{(\mathbb{R}^2}) } \right\rbrace $$
		%with the norm	
		%$$||u||^2=\int_{\mathbb{R}^2}\left[ |\nabla u|^2+u^2\right]dx.  $$	 
As shown in Sect. \ref{s2}, the functional $I$ is well-defined on $E$ and satisfies $I \in C^1(E, \mathbb{R})$. Moreover, 
for $u, \varphi \in E$,
\[
			\langle I'(u), \varphi \rangle = \int_{\mathbb{R}^2} \left[ \nabla u \nabla \varphi + V u \varphi -u \varphi \ln u^2\right] dx+ \lambda\int_{\mathbb{R}^2} \int_{\mathbb{R}^2} \ln |x-y| u^2(y)u(x)\varphi(x)dxdy\,.
			\]
            By a nontrivial solution of equation \eqref{1-4}, we mean a function $u\in E\setminus \{0\}$ such that $\langle I'(u), \varphi \rangle =0$ for any $\varphi \in E$. We say that $u_0\in E$ is a ground state solution of equation \eqref{1-4} if it is a least energy solution of equation \eqref{1-4}; namely, $u_0$ is a nontrivial solution and $I(u_0)\leq I(u)$ for every nontrivial solution $u$ of equation \eqref{1-4}.

	\vspace{.1cm}
	Our first result focuses on  the existence of ground state solution for equation \eqref{1-4} with $\lambda>0$. %To establish this result, we begin by defining the mountain pass value
	%$$c_{*,1}=\inf_{\gamma\in \Gamma}\max_{t\in [0,1]}I(\gamma (t)), \quad \mathrm{where}\ \Gamma=\{\gamma\in C([0,1],E)\ |\ \gamma(0)=0, \ I(\gamma(t))<0\}$$
    %and $$c_{**}:=\inf\{I(w)\ |\ w\in E\backslash\{0\} \ \mathrm{solves}\ \eqref{1-4} \}.$$
    
		\begin{theorem}\label{theorem1.1}
	Assume that $V$ satisfies $(V_1)$-$(V_2)$  and $\lambda>0$. Then, equation \eqref{1-4} has a ground state solution. %we have 
	%\\(i) $c_*>0 $ and equation \eqref{1-4} has a non-trivial solution $u\in E\backslash\{0\}$  satisfying $I(u)=c_*$.
	%\\(ii) Equation \eqref{1-4} has a ground state solution: a nontrivial solution  $v\in E\backslash\{0\}$  satisfying $I(v)=c_{**}$, where the ground state energy is
	%$$c_{**}:=\inf\{I(w)\ |\ w\in E\backslash\{0\} \ \mathrm{solves}\ \eqref{1-4} \}.$$	
	\end{theorem}

    The proof of Theorem \ref{theorem1.1} is based on the variational method. However, the presence of the logarithmic convolution makes the Nehari-type constrained minimization method used in \cite{S} inapplicable in our setting. Moreover, the logarithmic nonlinearity prevents us from using the Nehari-Pohozaev manifold method developed by Ruiz \cite{R} since we cannot find a valid fibering map. In contrast to \cite{WZZ}, where the functional is bounded from below, we show that the functional $I$ satisfies the mountain pass geometry. On the other hand, due to the appearance of logarithmic convolution, it is unknown whether an arbitrary Palais-Smale sequence is bounded in $E$. In order to overcome this difficulty, we follow an idea from \cite{J} and introduce an auxiliary functional $\Phi$ given by \eqref{a}. By applying a general minimax principle to the auxiliary functional $\Phi$, we obtain a Cerami sequence for $I$ with a key additional property, which plays an important role in establishing its boundedness. Then, we prove the existence of a nontrivial solution to equation \eqref{1-4}. Finally, by minimizing the functional $I$ over the set of all nontrivial solutions, we succeed in showing that equation \eqref{1-4} admits a ground state solution. In the course of the proof, the scaling parameter $t$ in the term $V(tx)$ cannot be directly extracted, necessitating a series of precise and careful computations.
    
  %  \noindent \textbf{Remark 1.1}: In the proof of Theorem \ref{theorem1.1},  we utilize preliminary definitions and lemmas in \cite{CW}, and it is inspired by the arguments in \cite{HIT,J,DW}. Since we investigated the planar Schr\"{o}dinger-Poisson  equation involving potential and logarithmic nonlinearity, which requires suitable assumptions on $V$ to  construct   Cerami sequence $\{u_n \}$. Moreover, the convergence $K(u_n)\rightarrow0$ as $n\rightarrow\infty$ is essential to establish the boundedness of $\{u_n \}$. Once the boundedness of the sequence is established, we combine $K(u_n)\rightarrow0$ with $E\hookrightarrow\hookrightarrow L^q(\mathbb{R}^2)$ for $2\leq q <2^*$ to rule out the trivial solution $u=0$. With the mountain-pass structure obtained, ${u_n}$ converges to a nontrivial solution $u$ to the problem \eqref{1-4}.

%One of the difficulties of this paper lies in proving that $I$ is well-defined in $E$ and $I \in C^1(E, \mathbb{R})$, since $u\ln u^2$ is sublinear near $0$ (see Sect. \ref{s2}).

%In the following, we  give the result of  the existence of a ground state solution for  problem \eqref{1-4} with $\lambda<0$.
\vspace{.1cm}
A natural question is whether equation \eqref{1-4} still has a ground state solution if $\lambda<0$. We proceed along the lines of the proof of Theorem \ref{theorem1.1}. However, the function $u_\theta$ defined in Lemma \ref{lemma2.7} and the auxiliary functional $\Phi$ defined by \eqref{a} are no longer applicable to the case $\lambda<0$. To overcome this difficulty, we construct a new function by \eqref{4.1} and a new auxiliary functional by \eqref{4.6}. Our result is as follows.
%Our next result is on this aspect.
		\begin{theorem}\label{theorem1.3}
		Assume that $V$ satisfies $(V_1)$-$(V_2)$ and $\lambda<0$. Then, equation \eqref{1-4} has a ground state solution.
	\end{theorem}
%We prove Theorem \ref{theorem1.3} along the lines of the proof of Theorem \ref{theorem1.1}. Since $\lambda<0$, the auxiliary functional $\Phi$ defined by \eqref{a} is no longer applicable. To overcome this difficulty, we construct another auxiliary functional, which is given by \eqref{4.6}.

    \vspace{.1cm}
	Our next result concerns the existence of sign-changing solutions to equation \eqref{1-4}. To this end, we introduce some notation.
 Let $G$ be a closed subgroup of the orthogonal group $O(2)$, and let $\varrho: G\rightarrow \{-1,1\}$ be a group homomorphism. Then, the pair $(G,\varrho)$ induces a group action of  $G$ on \( E \),  defined as follows:
	 \begin{align*}%\label{q1}
	 	[A\diamond u](x)=\varrho (A)u(A^{-1}x)\quad \mathrm{for}\ \ A\in G,\  u\in E \ \mathrm{and}\ x\in \mathbb{R}^2.
	 \end{align*}
	The subsequent result concerns solutions of equation \eqref{1-4} in the invariant space
	\begin{align*}%\label{q2}
		E_G:=\{u\in E\ |\ A\diamond u=u \quad \mathrm{for}\ \mathrm{all} \ \ A\in G\}.
	\end{align*}
%With this invariant space in hand, we proceed to define the critical value
%$$c_{G}=\inf_{\gamma\in \Gamma_G}\max_{s\in [0,1]}I(\gamma(s))\quad \mathrm{with}\ \Gamma_G:=\{\gamma\in \left( C[0,1], E_G\right), \gamma(0)=0, I(\gamma(1))<0 \}.$$

	\begin{theorem}\label{theorem1.2}
	Let $G, \varrho$ be as above. Suppose that $\lambda\neq 0$, $E_G\neq \{0\}$, and \(V: \mathbb{R}^2 \to \mathbb{R}\) satisfies $(V_1)$-$(V_2)$ and $V(Ax)=V(x)$ for any $A\in G$. Then,
	equation \eqref{1-4} has a G-invariant ground state solution, that is, a solution  $u\in E_{G}\backslash\{0\}$  satisfying $I(u)=c_{*,G}$, where 
	$$c_{*,G}:=\inf\{I(w)\ |\ w\in E_G\backslash\{0\} \ \mathrm{solves}\ \eqref{1-4} \}.$$
	\end{theorem}
    It is worth noticing that if $\varrho$ is nontrivial and $A\in G$ is given with $\varrho (A)=-1$, then every $u\in E_{G}$ satisfies $u(A^{-1}x)=-u(x)$. So, $u$ vanishes on the set $\{x\in \mathbb{R}^2\ |\ Ax=x\}$ and changes sign in $\mathbb{R}^2$ if $u\neq 0$.

\vspace{.1cm}
    Next, we introduce an example for $G$ and $\varrho$.
    
\noindent\textbf {Example 1.1.}  Fix $k\in\mathbb{N}$. Consider the subgroup $G$ of  $O(2)$ of order $2k$ generated by the (counter-clockwise)  $\frac{\pi}{k}$-rotation  
$$f\in O(2),\quad   f:=\begin{pmatrix}
\cos \frac{\pi}{k} & -\sin \frac{\pi}{k} \\
 \sin \frac{\pi}{k} & \cos \frac{\pi}{k}
\end{pmatrix},$$
then $$f x=\left(x_1\cos \frac{\pi}{k}-x_2 \sin \frac{\pi}{k}, x_1 \sin \frac{\pi}{k}+x_2\cos \frac{\pi}{k}\right), \quad \mathrm{for}\ \ x=(x_1,x_2)\in \mathbb{R}^2.$$
Let $\varrho:G\to\{-1,1\}$ be the group homomorphism defined by
\[
\varrho(f^i)=(-1)^i,\qquad i=1,2,\dots,2k,
\]
where $f^i$ is the $\frac{i\pi}{k}$-rotation. If \(V: \mathbb{R}^2 \to \mathbb{R}\) satisfies $(V_1)$-$(V_2)$ and $V(fx)=V(x)$, then it follows from Theorem \ref{theorem1.2} that equation \eqref{1-4} possesses a $G$-invariant solution, which is a sign-changing solution.

\vspace{.1cm}
If, in addition, $V$ is radially symmetric, we have the following corollary.
\begin{corollary}\label{corollary1}
	Assume that   $V$ is radially symmetric satisfying $(V_1)$-$(V_2)$ and $\lambda\neq 0$, then equation \eqref{1-4} admits a sequence $\{u_n\}$ of nonradial sign-changing solutions  with $I(u_n)\rightarrow +\infty$ as $n\rightarrow\infty$.
\end{corollary}

\noindent \textbf{Remark 1.2.}	
We provide several examples of potentials $V$ satisfying $(V_1)$-$(V_2)$.

$(i)$\ \ $V(x) = 1+m|x|^{\gamma}$ 
with $m>0$ and $\gamma>0$.

$(ii)$\ \ $V(x) = 1+m|x|^{\gamma}+\frac{1}{2}\sin x_1$ with $m>0$ and $\gamma>0$.
%$$V(x) = m\bigl(1+|x|^2\bigr)^{\frac{\gamma}{2}} + \frac{\varepsilon x_1}{\sqrt{1+|x|^2}},$$
%where $x=(x_1,x_2)\in\mathbb{R}^2$, $m>0$, $x_1>0$ and $0<\varepsilon<m$.
%\\ Note that potentials $V$ includes both radially symmetric potentials and general coercive potentials.

%\noindent \textbf{Remark 1.2}: We provide several examples of potentials $V$ satisfying $V_1-V_2$.
%\\(i) $V(x)=m(1+|x|^2)^{\frac{\gamma}{2}}$ with $m>0$ and $\gamma>2$.
%\\(ii) $V(x)=m(1+|x|^2)^{\frac{\gamma}{2}}+\frac{\varepsilon x_1}{\sqrt{1+x^2}}$ where $x\in (x_1,x_2)\in \mathbb{R}^2$, $m>0, x_1>0$ and $0<\varepsilon <m$.

    \vspace{.1cm}
	This paper is organized as follows.
In Sect. \ref{s2}, we set up the variational framework and give some useful lemmas. The proofs of Theorems \ref{theorem1.1} and \ref{theorem1.3} are given in Sect. \ref{s3} and Sect. \ref{s4}, respectively. Sect. \ref{s5} focuses on the proofs of Theorem \ref{theorem1.2} and Corollary \ref{corollary1}.
%first, we establish the existence and boundedness of Cerami sequences. Then, by the mountain pass theorem and Pohozaev identity, we give the results of  ground state solutions with $\lambda>0$. In Sect. 4, we derive the same conclusions as those obtained in Section 3 with $\lambda<0$. 

\vspace{.1cm}
Throughout the paper, we use the following notation for convenience.

$\bullet$ $\|u\|_{q}:=(\int_{\mathbb R^{2}}|u|^{q}dx)^{\frac{1}{q}}$ for $ 1\leq q<\infty$ and $u\in L^{q}(\mathbb R^{2})$.

$\bullet$ $ C, C_{i}$ denote various positive constants which may differ.

	\section{Preliminaries}\label{s2}
In this section, we set up the variational framework and give some useful lemmas.
First, we give some embedding results. 
			\begin{lemma}\label{lemma1}
				 Assume that $V$ satisfies $(V_1)$, then \\
                 (i) $E\hookrightarrow X$, where $X$ is given by \eqref{X}.\\
                 (ii) $E\hookrightarrow\hookrightarrow L^q(\mathbb{R}^2)$ for all $q\geq 2$.\\
                 (iii) $E\hookrightarrow L^{2-\varepsilon}(\mathbb{R}^2)$ for all $\varepsilon\in \left(0,\frac{2\gamma}{\gamma+2}\right)$.
		\end{lemma}
		\noindent   {\it Proof.} 
		(i) By $(V_1)$, we have $V_0:=\inf\limits_{x\in \mathbb{R}^2} V(x)>0$ and 
        \begin{align}\label{2.1}
           V_0 \int_{\mathbb{R}^2}u^2 dx
           \leq \int_{\mathbb{R}^2}V(x)u^2 dx\quad \mathrm{for\ all}\quad  u\in E.
        \end{align}
        Note that 
        \begin{align*}
        \lim_{|x| \to \infty } \frac{\ln(1+|x|)}{V(x)}
           = \mathop {\lim }\limits_{|x| \to \infty } \frac{\ln(1+|x|)}{|x|^\gamma}\cdot \lim_{|x| \to \infty } \frac{|x|^\gamma}{V(x)}=0,
        \end{align*}
        which implies that there exists $R_1 > 0$ such that 
        \begin{align*}
           \ln(1+|x|) \le V(x) \quad \mathrm{for}\quad  |x| \ge R_1.
        \end{align*}
        Denote $C_1:=\mathop {\max }\limits_{|x| \le R_1 }\frac{\ln(1+|x|)}{V(x)}$. Then, 
        \begin{align}\label{2.1.1}
            \ln(1+|x|)\leq \max\{C_1,1\}V(x)\quad \mathrm{for}\quad  x\in \mathbb{R}^2,
        \end{align}
        and moreover 
        \begin{align}\label{2.2}
            \int_{\mathbb{R}^2}\ln(1+|x|)u^2 dx
           \leq \max\{C_1,1\}\int_{\mathbb{R}^2}V(x)u^2 dx
           \quad \mathrm{for\ all}\quad  u\in E.
        \end{align}
        Combining \eqref{2.1} and \eqref{2.2}, we obtain that there exists $C> 0$ such that $\| u \|_X\leq C\| u \|_E$ for all $u\in E$, which means $E\hookrightarrow X$.
        
       (ii) % Due to $\inf_{x \in \mathbb{R}^2} V(x) > 0$, we have $E\hookrightarrow L^2(\mathbb{R}^2)$. 
       By (i) and \cite[Lemma 2.2]{CW}, we derive $E \hookrightarrow \hookrightarrow L^q(\mathbb{R}^2)$ for all $q \ge 2$.
        
		(iii) According to $(V_1)$, there exist $R_2 > 0$ and $C_0 > 0$ such that 
	\begin{align}\label{01}
		V(x) \ge C_0 |x|^\gamma \quad \mathrm{for}\quad  |x| \ge R_2.
	\end{align}
		Then, for $u \in E$, it follows from \eqref{01} that
		\begin{align}\label{02}
			\int_{\mathbb{R}^2} (1+|x|^\gamma) u^2 \, dx \notag
            &= \int_{|x| \le R_2} (1+|x|^\gamma) u^2 \, dx + \int_{|x| > R_2} (1+|x|^\gamma) u^2 \, dx \\
		\notag	
        %&\le C \int_{|x| \le R_2} u^2 \, dx + \int_{|x| > R_2} (1+|x|^\gamma) u^2 \, dx\\\notag
		&\le C \int_{|x| \le R_2} u^2\, dx+\int_{|x| > R_2} u^2\, dx+\frac{1}{C_0}\int_{|x| > R_2} V(x)u^2dx
		 \\\notag
			&\le C \int_{\mathbb{R}^2} u^2 \, dx + \frac{1}{C_0} \int_{\mathbb{R}^2} V(x) u^2 \, dx \\
			&\le C \| u \|_E^2,
		\end{align}
where $C>0$ is independent of $u$.	
		 For $\varepsilon\in \left( 0,\frac{2\gamma}{\gamma+2}\right) $, let
		$$\eta:=\frac{2-\varepsilon}{2},\quad p:=\frac{2}{2-\varepsilon}\quad \mathrm{and}\quad  \dfrac{1}{p}+\dfrac{1}{p'}=1.$$ For any $u\in E$, using the H\"older inequality and \eqref{02}, we obtain
		\begin{align}
			\int_{\mathbb{R}^2}|u|^{2-\varepsilon}dx\notag&=\int_{\mathbb{R}^2}\dfrac{1}{\left(1+|x|^\gamma\right)^\eta }\left(1+|x|^\gamma \right)^\eta|u|^{2-\varepsilon}dx \\\notag
            &\leq\left( \int_{\mathbb{R}^2}\dfrac{1}{\left( 1+|x|^\gamma\right)^{\eta p'} }dx\right )^{\frac{1}{p'}}\left(\int_{\mathbb{R}^2} \left[(1+|x|^\gamma)^\eta|u|^{2-\varepsilon} \right]^pdx \right)^{\frac{1}{p}}\\\notag
            &\leq C%\left( \int_{\mathbb{R}^2}\dfrac{1}{\left( 1+|x|^\gamma\right)^{\eta p'} }dx\right )^{\frac{1}{p'}}
            \left( \int_{\mathbb{R}^2} (1+|x|^\gamma)|u|^2dx\right) ^\frac{1}{p}\\\notag
           % & \leq C\left(\int_{\mathbb{R}^2} \dfrac{1}{(1+|x|^\gamma)^{\eta p^{'}}} \right)^{\frac{1}{p^{'}}}||u||_E^{2-\varepsilon}\\\notag
            &\leq C||u||_E^{2-\varepsilon},
		\end{align}
where we used $\gamma\eta p^{'}=\dfrac{\left( 2-\varepsilon\right)\gamma }{\varepsilon}>2$. Hence, $E\hookrightarrow L^{2-\varepsilon}(\mathbb{R}^2)$ for $\varepsilon\in \left(0,\frac{2\gamma}{\gamma+2}\right)$.
	\hfill$\square$			

\vspace{.2cm}
In order to treat the logarithmic convolution variationally, we define the following symmetric bilinear forms as in \cite{CW}:
	\begin{align*}
	D_1(u,v):=&\int_{\mathbb{R}^2}\int_{\mathbb{R}^2}\ln(1+|x-y|)u(x)v(y)dxdy,\\
    D_2(u,v):=&\int_{\mathbb{R}^2}\int_{\mathbb{R}^2}\ln  \left(1+\dfrac{1}{|x-y|} \right)u(x) v(y)dxdy.
	\end{align*} 
%	and 
%	\begin{align}
%	\notag	D_0(u,v):=\int_{\mathbb{R}^2}\int_{\mathbb{R}^2}\ln |x-y|u(x)v(y)dxdy.
%	\end{align}
For $u,v,w,z\in E$, we deduce from \cite[(2.4)]{CW} and \eqref{2.2} that 
		\begin{align}\label{w1}
|D_1(uv,wz)|\notag=&\left|\int_{\mathbb{R}^2}\int_{\mathbb{R}^2}\ln(1+|x-y|)u(x)v(x)w(y)z(y)dxdy\right|\\
\notag\leq &\int_{\mathbb{R}^2}\int_{\mathbb{R}^2}\left[\ln(1+|x|)+\ln(1+|y|)\right]|u(x)v(x)w(y)z(y)|dxdy\\
%&\leq\int_{\mathbb{R}^2}\int_{\mathbb{R}^2}\ln (1+|x|)u(x)v(x)w(y)z(y)dxdy+\int_{\mathbb{R}^2}\int_{\mathbb{R}^2}\ln (1+|y|)u(x)v(x)w(y)z(y)dxdy\\\notag
\notag\leq &||u||_2||v||_2\left(\int_{\mathbb{R}^2}\ln (1+|x|)w^2dx\right)^{\frac{1}{2}}\left(\int_{\mathbb{R}^2}\ln (1+|x|)z^2dx\right)^{\frac{1}{2}} \\\notag
&+||w||_2||z||_2\left(\int_{\mathbb{R}^2}\ln (1+|x|)u^2dx\right)^{\frac{1}{2}}\left(\int_{\mathbb{R}^2}\ln (1+|x|)v^2dx\right)^{\frac{1}{2}}\\\notag
\leq &C||u||_2||v||_2\left(\int_{\mathbb{R}^2} V(x)w^2(x)dx\right)^{\frac{1}{2}}\left(\int_{\mathbb{R}^2} V(x)z^2(x)dx\right)^{\frac{1}{2}}
\\
&+C||w||_2||z||_2\left(\int_{\mathbb{R}^2} V(x)u^2(x)dx\right)^{\frac{1}{2}}\left(\int_{\mathbb{R}^2} V(x)v^2(x)dx\right)^{\frac{1}{2}}.
	\end{align}
%    For $u,v\in L^{\frac{4}{3}}(\mathbb{R}^2)$, it follows from \cite[(2.1)]{CW} that 
Since $0<\ln (1+r)<r$ for $r>0$, it follows from the Hardy-Littlewood-Sobolev inequality \cite{25} that there exists a constant $C>0$ such that 
	\begin{align}\label{1-7}
|D_2(u,v)|
%\notag&\leq\int_{\mathbb{R}^2}\int_{\mathbb{R}^2}\ln  \left(1+\dfrac{1}{|x-y|} \right)|u(x) v(y)|dxdy\\\notag&
\leq\int_{\mathbb{R}^2}\int_{\mathbb{R}^2}\frac{1}{|x-y|}|u(x)v(y)|dxdy
\leq C||u||_{\frac{4}{3}}||v||_{\frac{4}{3}}\ \ \text{for}\ \ u,v\in L^{\frac{4}{3}}(\mathbb{R}^2).
\end{align}

Thus, we have the following lemma.
\begin{lemma}
    There exists $C>0$ such that for all $u\in E$, 
    \begin{align}
        \int_{\mathbb{R}^2}\int_{\mathbb{R}^2}\ln(1+|x-y|)u^2(x)u^2(y)dxdy
        \leq &C||u||_2^2\int_{\mathbb{R}^2} V(x)u^2(x)dx,\label{2.8.0}\\ 
        \int_{\mathbb{R}^2}\int_{\mathbb{R}^2}\ln  \left(1+\dfrac{1}{|x-y|} \right)u^2(x) u^2(y)dxdy
        \leq &C||u||_{\frac{8}{3}}^4.\label{2.8.1}
    \end{align}
\end{lemma}

  According to \cite[lemma 2.6]{CW}, we have the following property of the bilinear form \(D_1\).
		
		\begin{lemma}%{\rm(\cite[lemma 2.6]{22})}
        \label{lemma2.3}
					Let $\{u_n\}, \{v_n\} $ and $\{w_n\}$ be bounded sequences in $E$ such that $u_n\rightharpoonup u$ in $E$. Then, for every $z\in E$, we have $D_1(v_nw_n,z(u_n-u))\to 0$ as $n\rightarrow\infty$.
					%\begin{align}
					%\notag	\int_{\mathbb{R}^2}\int_{\mathbb{R}^2}\ln (1+|x-y|)v_n(y)w_n(y)w(x)(u_n(x)-u(x))dxdy\rightarrow0, \quad \mathrm{as}\quad n\rightarrow\infty.
					%\end{align}
		\end{lemma}
        
	Next, we give some properties of the logarithmic nonlinearity $s^2\ln s^2$.
    \begin{lemma}\label{lemmaa} %The following conclusions hold.
(i) For any $\varepsilon \in (0,2)$, there exists $C_\varepsilon>0$ such that $|s^2\ln s^2|\leq C_\varepsilon |s|^{2-\varepsilon} $ for $|s|\leq 1.$
\\(ii) For any $\delta >0$, there exists $C_\delta>0$ such that $|s^2\ln s^2|\leq C_\delta |s|^{2+\delta} $ for $|s|\geq1.$
\end{lemma}	
\noindent   {\it Proof.} $(i)$ Since 
\begin{align*}
    \displaystyle \lim_{s \to 0^+} \frac{s^2 \ln s^2}{s^{2-\varepsilon}} = \lim_{s \to 0^+} s^{\varepsilon} \ln s^2 = 0,
\end{align*}
 there exists $M_1 \in (0,1)$ such that 
$|s^2\ln s^2| \le s^{2-\varepsilon}$ for $s \in (0, M_1)$. Denote 
$C_1 = \max_{\substack{s \in [M_1,1]}} \left| {s^\varepsilon \ln s^2} \right|$ and $C_\varepsilon = \max\{1, C_1\}$, then we obtain  $|s^2 \ln s^2| \le C_\varepsilon s^{2-\varepsilon}$ for all $0 < s \le 1$.
\\$(ii)$ Since 
\begin{align*}
    \displaystyle \lim_{s \to+ \infty} \frac{s^2 \ln s^2}{s^{2+\delta}} = \lim_{s \to +\infty} \frac{\ln s^2}{s^{\delta}} = 0,
\end{align*}
there exists $M_2 > 1$ such that 
$|s^2 \ln s^2| \le s^{2+\delta}$ for $s \in (M_2, +\infty)$.
Denote 
$C_2 = \max_{\substack{s \in [1,M_2]}} \left|s^{-\delta}\ln s^2 \right|$ and $C_\delta = \max\{1, C_2\}$, then we have $|s^2\ln s^2| \le C_\delta s^{2+\delta}$ for all $s\ge 1$.
\hfill$\square$	

\vspace{.2cm}
Having established the above lemmas, we are able to study the regularity of the functional $I$ on $E$.
		\begin{lemma}\label{lemma 2}
			The functional $I$ is well-defined on $E$ and $I \in C^1(E, \mathbb{R})$. Moreover, for $u, \varphi \in E$,\[
			\langle I'(u), \varphi \rangle = \int_{\mathbb{R}^2} \left[ \nabla u \nabla \varphi + V u \varphi \right] dx+ \lambda\int_{\mathbb{R}^2} \int_{\mathbb{R}^2} \ln |x-y| u^2(y)u(x)\varphi(x)dxdy- \int_{\mathbb{R}^2} u \varphi \ln u^2dx.		\]
		\end{lemma}	
		\noindent   {\it Proof.} 
		 Define
		\[
		I_1(u) := \frac{1}{2} \int_{\mathbb{R}^2} \left[ |\nabla u|^2 + V(x) u^2 \right] dx+ \frac{\lambda}{4} \int_{\mathbb{R}^2} \int_{\mathbb{R}^2} \ln |x-y|u^2(x)u^2(y)dxdy,
		\]
		\[
		I_2(u) := \int_{\mathbb{R}^2} \left( u^2\ln u^2 - u^2 \right)dx.
		\]
Similar to \cite[Lemma 2.2]{CW}, we obtain that $I_1$ is well-defined on $E$ and $I_1\in C^1(E, \mathbb{R})$. Moreover,

 $$\langle I_1'(u), \varphi \rangle = \int_{\mathbb{R}^2}\left[  \nabla u  \nabla \varphi + V u \varphi \right] dx+\lambda \int_{\mathbb{R}^2} \int_{\mathbb{R}^2} \ln |x-y| u^2(y)u(x)\varphi(x) dxdy.$$

Next, we prove that $I_2$ is well-defined on $E$ and $I_2 \in C^1(E, \mathbb{R})$.
Denote $\varepsilon_0 := \frac{\gamma}{\gamma+2}$. It follows from Lemma \ref{lemmaa} that
\begin{align}\label{2.8}
    |u^2 \ln u^2|\leq C\left( |u|^{2-\varepsilon_0} + |u|^3 \right).
\end{align}
We deduce from \eqref{2.8} and Lemma \ref{lemma1} that
$$\left| \int_{\mathbb{R}^2} u^2 \ln u^2 dx\right| \leq C \int_{\mathbb{R}^2} \left( |u|^{2-\varepsilon_0} + |u|^3 \right)dx \leq C \left( \|u\|_{E}^{2-\varepsilon_0}+ \|u\|_{E}^3 \right) < +\infty,$$
which means that $I_2$ is well-defined on $E$. Suppose that $u_n \to u$ in $E$. By Lemma \ref{lemma1}, we get $u_n \to u$ in $L^{2-\varepsilon_0}(\mathbb{R}^2) \cap L^3(\mathbb{R}^2)$. According to \eqref{2.8} and \cite[Theorem A.4]{W}, we obtain 
\begin{align*}
    u_n^2 \ln u_n^2 \to u^2\ln u^2\ \ \text{in}\ \ {L}^1(\mathbb{R}^2),
\end{align*}
%$$u_n^2 \ln u_n^2 \to u^2\ln u^2,$$ 
which implies
\[
\left| \int_{\mathbb{R}^2} \left( u_n^2 \ln u_n^2 - u^2 \ln u^2 \right)dx \right| \leq \| u_n^2 \ln u_n^2 - u^2 \ln u^2\|_{{L}^1(\mathbb{R}^2)} \to 0 \, \quad \mathrm{as}\quad n\rightarrow \infty
.
\]
So, 
$I_2 \in C(E, \mathbb{R})$.

For 
$u, \varphi \in E$, $0 < |t| < 1$, by the mean value theorem, there exists $\theta \in (0,1)$
such that
\begin{align}
\notag&\frac{1}{|t|} \left| (u+t\varphi)^2 \ln (u+t\varphi)^2 - \left(u+t\varphi \right)^2  - \left(u^2\ln u^2-u^2 \right) \right|\\\notag
 = &2 |\varphi (u+\theta t \varphi) \ln (u+\theta
t \varphi)^2|
\\\notag
\leq &C |\varphi| \left( |u+\theta t \varphi|^{1-\varepsilon_0} + |u+\theta t \varphi|^2 \right)\\\notag
\leq &C |\varphi| \left( |u|^{1-\varepsilon_0} + |\varphi|^{1-\varepsilon_0} + |u|^2 + |\varphi|^2 \right) \in L^1(\mathbb{R}^2).
\end{align}
It follows from Lebesgue's dominated convergence theorem that
\begin{align*}
    \langle I_2'(u), \varphi \rangle
    = &\lim_{t \to 0} \frac{1}{t} \left[ I_2(u+t\varphi) - I_2(u) \right] = \lim_{t \to 0} 2\int_{\mathbb{R}^2} \varphi(u+\theta t \varphi) \ln (u+\theta t \varphi)^2dx \\
    =&2 \int_{\mathbb{R}^2} \varphi u \ln u^2dx.
\end{align*}
Suppose that 
$u_n \to u$ in $E$ as $n\rightarrow\infty$, then $u_n \to u$ in $L^{2-\varepsilon_0}(\mathbb{R}^2) \cap L^3(\mathbb{R}^2)$. Note that %$|u \ln u^2| \leq C \left( |u|^{1-\frac{\varepsilon_0}{2}} + |u|^\frac{3}{2} \right)$,
\begin{align*}
    |u \ln u^2| \leq C \left( |u|^{1-\frac{\varepsilon_0}{2}} + |u|^\frac{3}{2} \right),
\end{align*}
which, together with \cite[Theorem A.4]{W}, yields that
$u_n \ln u_n^2 \to u \ln u^2$ in $L^2(\mathbb{R}^2)$ as $n\rightarrow\infty$. Therefore,
\begin{align}
	\left| \langle I_2'(u_n), \varphi \rangle - \langle I_2'(u), \varphi \rangle \right|
\notag 
&	= \left|2 \int_{\mathbb{R}^2} \left( u_n \ln u_n^2 - u \ln u^2 \right) \varphi  dx \right|
	\\\notag
    &\leq 2\|u_n \ln u_n^2 - u \ln u^2\|_{L^2(\mathbb{R}^2)} \|\varphi\|_{L^2(\mathbb{R}^2)}
\\&	= o_n(1) \cdot \|\varphi\|_E,\notag
\end{align}
which implies that $\|I_2'(u_n) - I_2'(u)\|_{E'} \to 0$. Here $E'$ denotes the dual space of $E$. To sum up, $I_2 \in C^1(E, \mathbb{R})$. We complete the proof. 
\qed

  \vspace{.2cm}
In the following, we recall the general minimax principle from \cite{3}, which will be used to prove the existence of a Cerami sequence for the functional $I$ with an additional property. %a slightly stronger variant of Theorem 2.8 in \cite{5} that yields Cerami sequences rather than Palais–Smale sequences.
		\begin{lemma}{\rm(\cite[Proposition 2.8]{3})}\label{lemma2.5}
		Let $Y$ be a Banach space. Let $N_0$ be a closed subspace of the metric space $N$ and $\Gamma_0\subset C(N_0, Y)$. Define
		$$\Gamma:=\{\gamma\in C(N,Y)\ | \ \gamma|_{N_0}\in \Gamma_0\}.$$
		If $\psi\in C^1(Y,\mathbb{R})$ satisfies 
		$$\infty>c:=\inf_{\gamma\in \Gamma}\sup_{u\in N}\psi(\gamma(u))>b:=\sup_{\gamma_0\in \Gamma_0}\sup_{u\in N_0} \psi(\gamma_0(u)),$$
		then, for every \(\varepsilon \in(0, \frac{c-b}{2})\), \(\delta>0\) and \(\gamma \in \Gamma\) with \(\sup _{u \in N} \psi(\gamma(u)) \leq c+\varepsilon\),  there exists \(u \in Y\) such that
	\\	\  $ (i)\ c-2 \varepsilon \leq \psi(u) \leq c+2 \varepsilon,$
	\\	\ $  (ii)\ dist (u, \gamma(N)) \leq 2 \delta$,
    \\	$	(iii)\ \left(1+\| u\| _{Y}\right)\left\| \psi'(u)\right\| _{Y'} \leq \frac{8 \varepsilon}{\delta}.$
\end{lemma}

Similar to \cite[Proposition 2.3]{CW}, we have the following regularity result for the solution of equation \eqref{1-4}.
\begin{lemma} \label{lemma2.13}
Let $u$ be a weak solution  of equation \eqref{1-4}. Then $u\in C^{2}(\mathbb{R}^2)$  and 
$$\lim_{|x|\rightarrow\infty}u(x)=0.$$
\end{lemma}
		
%		\begin{lemma}\label{lemma2.8}
%Assume that $u\in E$ is a weak solution to \eqref{1-4}, then the following Pohoz\v{a}ev type identity holds:
%\end{lemma}	
%\noindent   {\it Proof.} For every $u\in E$, define $u_t:=u\left(\frac{x}{t} \right) $, we deduce
%\begin{align}
%		I(u_t)\notag&=\dfrac{1}{2}\int_{\mathbb{R}^2}\left[ |\nabla u_t|^2+u_t^2+V(x)u_t^2\right]dx -\frac{1}{2}\int_{\mathbb{R}^2}u_t^2\ln u_t^2dx\\\notag&\quad +\frac{\lambda}{4}\int_{\mathbb{R}^2}\int_{\mathbb{R}^2} \ln |x-y|u_t^2(x)u_t^2(y)dxdy\\\notag&=\dfrac{1}{2}\int_{\mathbb{R}^2}|\nabla u|^2dx+\dfrac{t^2}{2}\int_{\mathbb{R}^2}u^2dx+\dfrac{t^2}{2}\int_{\mathbb{R}^2} V(tx)u^2dx-\dfrac{t^2}{2}\int_{\mathbb{R}^2} u^2\ln u^2dx\\\notag&\quad +\dfrac{\lambda t^4\ln t}{4}\left( \int_{\mathbb{R}^2} u^2dx\right) ^2+\dfrac{\lambda t^4}{4}\int_{\mathbb{R}^2} \ln |x-y|u^2(x)u^2(y)dxdy,
%\end{align}
%thus		
%\begin{align}
%	\dfrac{d}{dt}I(u_t)|_{t=1}\notag&=\int_{\mathbb{R}^2}u^2dx+\int_{\mathbb{R}^2} V(x)u^2dx+\dfrac{1}{2}\int_{\mathbb{R}^2}(\nabla V\cdot x) u^2dx\\\notag&\quad +\lambda\int_{\mathbb{R}^2}\int_{\mathbb{R}^2} \ln|x-y|u^2(x)u^2(y)dxdy+\dfrac{\lambda}{4}\left(\int_{\mathbb{R}^2} u^2dx \right)^2-\int_{\mathbb{R}^2}u^2\ln u^2dx,
%\end{align}
%which means \eqref{r1} holds. \hfill$\square$

%Moreover, using Lemma \ref{lemma2.8}, we define 

        \section{Ground state solution for $\lambda>0$}\label{s3}	

In this section, we will give the proof of Theorem \ref{theorem1.1}. % by Pohozaev identity and Lemma \ref{lemma2.5}.  		
First, we show that the energy functional $I$ satisfies the mountain pass geometry for $\lambda>0$.

		\begin{lemma}\label{lemma2.6}
	There exists $\alpha>0$ such that for all $0<\tau\le\alpha$,
	$$ m_{\tau}:=\inf \{I(u)\ |\ u\in E, ||u||_E=\tau\}>0$$
	and
	$$e_{\tau}:=\inf \{\langle I'(u),u\rangle \ | \ u\in E, ||u||_E=\tau\}>0. $$
	\end{lemma}				
	\noindent   {\it Proof.} 
	For $q>2$ and $u\in E$, we deduce from Lemmas \ref{lemma1} and \ref{lemmaa} that 
	\begin{align}\label{c5}
		\int_{\mathbb{R}^2} u^2\ln u^2dx\leq \int_{\{|u| > 1\}} u^2\ln u^2dx \leq C\int_{\{|u| > 1\}} |u|^qdx\leq C||u||_q^q\leq C||u||_E^q,
	\end{align} 
	where $C>0$ is independent of $u$. Moreover, 
	%for every $u\in E$,	by the definition $I$, 	
    by \eqref{2.8.1}, \eqref{c5} and Lemma \ref{lemma1}, we get
	\begin{align}
	I(u)\notag%&=\dfrac{1}{2}\int_{\mathbb{R}^2}\left[ |\nabla u|^2+u^2+V(x)u^2\right]dx -\frac{1}{2}\int_{\mathbb{R}^2}u^2\ln u^2dx\\\notag
   % &\quad +\frac{\lambda}{4}\int_{\mathbb{R}^2}\int_{\mathbb{R}^2} \ln |x-y|u^2(x)u^2(y)dxdy\\\notag
    &=\dfrac{1}{2}\int_{\mathbb{R}^2}\left[ |\nabla u|^2+u^2+V(x)u^2\right]dx -\frac{1}{2}\int_{\mathbb{R}^2}u^2\ln u^2dx\\\notag&\quad +\frac{\lambda}{4}\int_{\mathbb{R}^2}\int_{\mathbb{R}^2} \ln \left( 1+|x-y|\right) u^2(x)u^2(y)dxdy\\\notag&\quad-\frac{\lambda}{4}\int_{\mathbb{R}^2}\int_{\mathbb{R}^2} \ln \left( 1+\dfrac{1}{|x-y|}\right) u^2(x)u^2(y)dxdy\\\notag
    &\geq\dfrac{1}{2}||u||_E^2-C||u||^q_E-C||u||_{\frac{8}{3}}^4\\\notag
    & \geq %\dfrac{1}{2}||u||^2-C_1||u||^q_q-C_2||u||^4=
    \frac{||u||^2_E}{2}\left(1-C||u||^2_E-C||u||^{q-2}_E \right).
	\end{align}
%where $C_1, C_2>0$ are constants and $2< q<2^*$. 
This means that there exists $\alpha>0$ such that $m_{\tau}>0$ for $0<\tau\leq \alpha$.	
Similarly, we can choose $\alpha>0$ small enough such that $e_{\tau}>0$ for $0<\tau\leq \alpha$. This completes the proof.
%Moreover
%\begin{align}
%\langle I'(u),u\rangle \notag&=\int_{\mathbb{R}^2}\left[ |\nabla u|^2+V(x)u^2\right]dx-\int_{\mathbb{R}^2}u^2\ln u^2dx\\\notag&\quad +\lambda\int_{\mathbb{R}^2}\int_{\mathbb{R}^2}\ln |x-y| u^2(x)u^2(y)dxdy\\\notag &\geq\int_{\mathbb{R}^2}\left[ |\nabla u|^2+V(x)u^2\right]dx-\int_{\mathbb{R}^2}u^2\ln u^2dx\\\notag&\quad -\lambda\int_{\mathbb{R}^2}\int_{\mathbb{R}^2}\ln \left( 1+\dfrac{1}{|x-y|}\right)  u^2(x)u^2(y)dxdy\\\notag&\geq\int_{\mathbb{R}^2} |\nabla u|^2dx+\inf V(x)\int_{\mathbb{R}^2}u^2dx-\int_{\mathbb{R}^2}u^2\ln u^2dx\\\notag&\quad -\lambda\int_{\mathbb{R}^2}\int_{\mathbb{R}^2}\ln \left( 1+\dfrac{1}{|x-y|}\right)  u^2(x)u^2(y)dxdy\\\notag&\geq C||u||^2-C_1||u||^q-C_2||u||^4=||u||^2\left(C-C_1||u||^{q-2}-C_2||u||^2 \right),
%\end{align}
%where $C_1, C_2>0$ are constants and $2< q<2^*$. This means that $e_{\tau}>0$ for $\alpha>0$ small enough.
\hfill$\square$
		
\begin{lemma}\label{lemma2.7}
Assume that  $u\in E\setminus \{0\}$, then %there exists $u_\theta\in E$ such that 
$$I(u_\theta)\rightarrow-\infty, \quad \mathrm{as} \ \ \theta\rightarrow+\infty,$$
where  $u_\theta(x):=\theta^{a}u(\theta x)$ and $a>\max\{2,1+\frac{\kappa}{2}\}$  for  $\theta>0$.
\end{lemma}					
\noindent   {\it Proof.} 
First, we claim $u_\theta \in E$ for $u\in E\setminus \{0\}$. Since $V$ satisfies $(V_2)$, we deduce from \cite[lemma 3.1]{LM1} that 
\begin{align}\label{3.2}
    V(\lambda x)\leq\max\{\lambda^{\kappa}, \lambda^{-\kappa}\}V(x)\ \ \text{for}\ \ \lambda >0,\ x\in \mathbb{R}^2.
\end{align}
So,
\begin{align}
	 \int_{\mathbb{R}^2} |\nabla u_\theta|^2dx=&\theta ^{2a}\int_{\mathbb{R}^2}|\nabla u|^2dx<+\infty,\label{c1}\\
     \int_{\mathbb{R}^2} V(x)u_\theta^2dx=&\theta^{2a-2}\int_{\mathbb{R}^2}V\left( \frac{x}{\theta}\right) u^2dx
     \leq \max\{\theta^{2a-2+\kappa},\theta^{2a-2-\kappa}\} \int_{\mathbb{R}^2} V(x)u^2dx<+\infty.\label{c2}
\end{align}
%\begin{align}
%	\int_{\mathbb{R}^2} V(x)u_\theta^2dx=\theta^2\int_{\mathbb{R}^2}V\left( \frac{x}{\theta}\right) u^2dx\leq \theta^{2+\kappa} \int_{\mathbb{R}^2} V(x)u^2dx<+\infty.
%\end{align}
%and
%\begin{align}
%	\int_{\mathbb{R}^2} \ln (1+|x|^2)u_\theta^2dx\notag&=\theta^2\int_{\mathbb{R}^2}\ln \left(1+\dfrac{|y|^2}{\theta^2} \right) u^2(y)dy\\\notag&=\theta^2\int_{\mathbb{R}^2} \ln\left(1+|y|^2\right)u^2dy+\theta^2\int_{\mathbb{R}^2} \left[\ln \left(1+\dfrac{|y|^2}{\theta^2} \right)  -\ln (1+|y|^2)\right] u^2dy.
%\end{align}
%Applying the Newton-Leibmiz formula, we obtain 
%\begin{align}
%\notag	\left|\ln \left( 1+\frac{|y|^2}{\theta^2}\right)-\ln (1+|y|^2)\right |=\ \left|\int_1^{\frac{1}{\theta^2}}
%\frac{|y|^2}{1+t|y|^2}dt\right|\leq 2|\ln \theta|,
%\end{align}
%which implies $$\left|\int_{\mathbb{R}^2}\left[\ln \left(1+\dfrac{|y|^2}{\theta^2} \right)   -\ln (1+|y|^2)\right]u^2 \right|\leq 2|\ln \theta|\int_{\mathbb{R}^2} u^2dx<+\infty.$$
%Hence, 
%\begin{align}\label{c4}
%	\int_{\mathbb{R}^2} \ln (1+|x|^2)u_\theta^2<\infty.
%\end{align}
It follows from \eqref{c1} and \eqref{c2} that $u_\theta \in E$.
For $\theta>1$, we have 
\begin{align}
	I(u_\theta)\notag&=\dfrac{1}{2}\int_{\mathbb{R}^2}\left[ |\nabla u_\theta|^2+u_\theta^2+V(x)u_\theta^2\right]dx -\frac{1}{2}\int_{\mathbb{R}^2}u_\theta^2\ln u_\theta^2dx\\
    \notag&\quad +\frac{\lambda}{4}\int_{\mathbb{R}^2}\int_{\mathbb{R}^2} \ln |x-y|u_\theta^2(x)u_\theta^2(y)dxdy\\\notag
  %  &=\dfrac{\theta^{2a}}{2}\int_{\mathbb{R}^2}|\nabla u|^2dx+\dfrac{\theta^{2a-2}}{2}\int_{\mathbb{R}^2} u^2dx+\dfrac{\theta^{2a-2}}{2}\int_{\mathbb{R}^2} V\left( \frac{x}{\theta}\right) u^2dx\\\notag&\quad+\dfrac{\lambda\theta^{4a-4}}{4} \int_{\mathbb{R}^2}\int_{\mathbb{R}^2}\ln |x-y| u^2(x)u^2(y)dxdy\\\notag&\quad -\dfrac{\lambda\theta^{4a-4}\ln \theta}{4}\left(\int_{\mathbb{R}^2} u^2dx \right)^2 -a\theta^{2a-2}\ln \theta\int_{\mathbb{R}^2} u^2dx-\frac{\theta^{2a-2}}{2}\int_{\mathbb{R}^2}u^2\ln u^2dx\\
    \notag&\leq\dfrac{\theta^{2a}}{2}\int_{\mathbb{R}^2}|\nabla u|^2dx+\dfrac{\theta^{2a-2}}{2}\int_{\mathbb{R}^2} u^2dx+\dfrac{\theta^{2a-2+\kappa}}{2}\int_{\mathbb{R}^2} V(x)u^2dx\\\notag&\quad+\dfrac{\lambda\theta^{4a-4}}{4} \int_{\mathbb{R}^2}\int_{\mathbb{R}^2}\ln |x-y| u^2(x)u^2(y)dxdy\\\notag&\quad -\dfrac{\lambda\theta^{4a-4}\ln \theta}{4}\left(\int_{\mathbb{R}^2} u^2dx \right)^2 -a\theta^{2a-2}\ln \theta\int_{\mathbb{R}^2} u^2dx-\frac{\theta^{2a-2}}{2}\int_{\mathbb{R}^2}u^2\ln u^2dx. %\\\notag&\rightarrow-\infty, \quad \mathrm{as} \quad \theta\rightarrow\infty.
\end{align}	
Note that $ 4a-4>2a-2+\kappa>0$, which means $I(u_\theta)\rightarrow -\infty$ as $\theta\rightarrow+\infty$. This completes the proof.
	\hfill$\square$

    \vspace{.2cm}
	According to Lemma \ref{lemma2.7}, we can define the mountain pass value
    $$c_{*,1}=\inf_{\gamma\in \Gamma}\max_{t\in [0,1]}I(\gamma(t)),$$ 
where %$\Gamma:= \left\lbrace \gamma\in C([0,1],E)| \ \gamma(0)=0,I(\gamma(1))<0\right\rbrace \not=\emptyset$. 
\begin{align}\label{3.4}
    \Gamma:= \left\lbrace \gamma\in C([0,1],E)\ | \ \gamma(0)=0,I(\gamma(1))<0\right\rbrace \not=\emptyset.
\end{align}
Using Lemma \ref{lemma2.6}, we derive $c_{*,1}\geq m_\alpha>0$.	

\vspace{.2cm}
However, due to the appearance of logarithmic convolution, it is unknown whether an arbitrary Palais-Smale sequence is bounded $E$. In order to overcome this difficulty, inspired by \cite{J}, we are going to construct a Cerami sequence with an additional property that is key to proving its boundedness. For this purpose, we consider the following Banach space
$$\tilde E:=\mathbb{R}\  \times E  $$		
equipped with the standard norm $||(b,w)||_{\tilde E}=\left( |b|^2+||w||_E^2\right)^{\frac{1}{2}} $ for $b\in \mathbb{R}$ and $ w\in E$. Furthermore, we define the  map $\eta: \tilde E\rightarrow E$ by
$$\eta(b,w)[x]:=e^{ab}w(e^bx), \quad \mathrm{for }\  b\in \mathbb{R}, w\in E \ \mathrm{and}\  x\in\mathbb{R}^2$$
and the functional $\Phi:\tilde E\rightarrow \mathbb{R}$ by
\begin{align}\label{a}
	\Phi(b,w):=&I(\eta(b,w))\nonumber\\
    \notag=&\dfrac{e^{2ab}}{2}\int_{\mathbb{R}^2} |\nabla w|^2dx+\dfrac{e^{2b(a-1)}}{2}\int_{\mathbb{R}^2}V\left(\frac{x}{e^b} \right) w^2dx+\dfrac{e^{2b(a-1)}}{2}\int_{\mathbb{R}^2}w^2dx\\
    \notag&+\dfrac{\lambda e^{4b(a-1)}}{4}\int_{\mathbb{R}^2}\int_{\mathbb{R}^2}\ln |x-y|w^2(x)w^2(y)dxdy-\dfrac{\lambda be^{4b(a-1)}}{4}\left( \int_{\mathbb{R}^2} w^2dx\right) ^2\\
    &-abe^{2b(a-1)}\int_{\mathbb{R}^2} w^2dx-\frac{e^{2b(a-1)}}{2}\int_{\mathbb{R}^2} w^2\ln w^2dx.
\end{align}	

\begin{lemma}\label{lemma2.15}
It holds $\Phi\in C^1{(\tilde E,\mathbb{R})}$. Furthermore,
\begin{align}\label{2-13}
	\Phi_b(b,w)\notag&=ae^{2ab}\int_{\mathbb{R}^2}|\nabla w|^2dx+(a-1)e^{2b(a-1)}\int_{\mathbb{R}^2}V\left(\dfrac{x}{e^b} \right) w^2dx\\
    \notag&\quad-\frac{e^{2b(a-1)}}{2}\int_{\mathbb{R}^2}\left( \nabla V\left( \frac{x}{e^b}\right)\cdot \frac{x}{e^b} \right) w^2dx\\
    \notag&\quad +\lambda (a-1)e^{4b(a-1)}\int_{\mathbb{R}^2}\int_{\mathbb{R}^2}\ln|x-y|w^2(x)w^2(y)dxdy\\\notag&\quad-\left(\lambda b(a-1)e^{4b(a-1)}+\dfrac{\lambda e^{4b(a-1)}}{4} \right)\left( \int_{\mathbb{R}^2}w^2dx\right)^2\\
&\quad-\left(e^{2b(a-1)}+2ab(a-1)e^{2b(a-1)} \right) \int_{\mathbb{R}^2} w^2dx-(a-1)e^{2b(a-1)}\int_{\mathbb{R}^2} w^2\ln w^2dx  
 %   \notag&=2\int_{\mathbb{R}^2}|\nabla \eta|^2dx +\int_{\mathbb{R}^2} V \eta^2dx-\dfrac{1}{2}\int_{\mathbb{R}^2}\left( \nabla V\cdot x\right) \eta^2dx\\
%\notag&\quad+\lambda\int_{\mathbb{R}^2}\int_{\mathbb{R}^2}\ln|x-y|\eta^2(x)\eta^2(y)dxdy -\dfrac{\lambda}{4}\left( \int_{\mathbb{R}^2}\eta^2dx    \right)^2-\int_{\mathbb{R}^2}\eta^2\ln \eta^2dx-\int_{\mathbb{R}^2} \eta^2dx \\
%    &=2I'(\eta)\eta-P(\eta)=K(\eta)
\end{align}		
and 
	\begin{align}\label{2-14}
	\langle \Phi_w(b,w),v\rangle \notag&=e^{2ab}\int_{\mathbb{R}^2}\nabla v\nabla wdx+e^{2b(a-1)}\int_{\mathbb{R}^2}V\left(\frac{x}{e^b} \right)vwdx-2abe^{2b(a-1)}\int_{\mathbb{R}^2}vwdx \\\notag&\quad +\lambda e^{4b(a-1)}\int_{\mathbb{R}^2}\int_{\mathbb{R}^2}\ln |x-y|w^2(y)w(x)v(x)dxdy\\\notag&\quad-\lambda be^{4b(a-1)}\int_{\mathbb{R}^2}w^2dx\int_{\mathbb{R}^2}vwdx-e^{2b(a-1)}\int_{\mathbb{R}^2}vw\ln w^2dx\\&=\langle I'(\eta(b,w)),\eta(b,v)\rangle ,
\end{align}		
$	\mathrm{for}\ \mathrm{all}\ b\in \mathbb{R}\ \mathrm{and}\ v,w\in E	$.
\end{lemma}	
\noindent   {\it Proof.} %By  $(V_1)$, we obtain 
Since the functional $I$ is well-defined on $E$, the functional $\Phi$ is well-defined on $\tilde E$.

In the following, we will prove $\Phi\in C( \tilde E,\mathbb{R})$.
By Lemma \ref{lemma 2}, we only need to prove %$\int_{\mathbb{R}^2} V\left(\frac{x}{e^{b}} \right)w^2dx \in C( \tilde E, \mathbb{R})$.
\begin{align}\label{3.7}
    \int_{\mathbb{R}^2} V\left(\frac{x}{e^{b}} \right)w^2dx \in C( \tilde E, \mathbb{R}).
\end{align}
Assume $(b_n, w_n)\rightarrow (b,w)$ in $ \tilde E$ as $n\rightarrow\infty$. It suffices to show that 
\begin{align}\label{a1}
	\int_{\mathbb{R}^2} V\left(\frac{x}{e^{b_n}} \right)w_n^2dx \rightarrow	\int_{\mathbb{R}^2} V\left(\frac{x}{e^{b}} \right)w^2dx, \quad \mathrm{as}\quad n\rightarrow\infty.
\end{align}
Note that 
\begin{align}\label{a2}
	\notag&\left| 	\int_{\mathbb{R}^2} V\left(\frac{x}{e^{b_n}} \right)w_n^2dx -	\int_{\mathbb{R}^2} V\left(\frac{x}{e^{b}} \right)w^2dx\ \right|\\
    %&=\left|\int_{\mathbb{R}^2} V\left(\frac{x}{e^{b_n}}\right)\left(w_n^2-w^2 \right)dx-\int_{\mathbb{R}^2} w^2\left[ V\left(\frac{x}{e^{b_n}} \right)-V\left(\frac{x}{e^{b}} \right)\right] dx \ \right|\\
    \leq &\left|\int_{\mathbb{R}^2} V\left(\frac{x}{e^{b_n}}\right)\left(w_n^2-w^2 \right)dx\ \right|+\left|\int_{\mathbb{R}^2} w^2\left[ V\left(\frac{x}{e^{b_n}} \right)-V\left(\frac{x}{e^{b}} \right)\right]dx\ \right|.
\end{align}
On the one hand, by  \eqref{3.2} and the H\"{o}lder inequality, we obtain
\begin{align}%\label{a3}
	\left|\int_{\mathbb{R}^2} V\left(\frac{x}{e^{b_n}}\right)\left(w_n^2-w^2 \right)dx\right|\notag
    &\leq \max\{e^{\kappa b_n }, e^{-\kappa b_n}\} \int_{\mathbb{R}^2}V(x)\left|\left(w_n-w \right)\left(w_n+w \right) \right| dx\\\notag
    &\leq C \left( \int_{\mathbb{R}^2}V(x)\left(w_n-w \right)^2dx\right)^{\frac{1}{2}}\left( \int_{\mathbb{R}^2} V(x)\left( w_n+w\right)^2dx\right)^{\frac{1}{2}}\\
    &\leq C||w_n-w||_E ||w_n+w||_E=o_n(1).
\end{align}
On the other hand, using  \eqref{3.2}, we get 
\begin{align}
	\left|w^2\left[V\left(\frac{x}{e^{b_n}} \right)-V\left(\frac{x}{e^{b}} \right) \right]\ \right|\notag&\leq w^2V\left(\frac{x}{e^{b_n}} \right) +w^2V\left( \frac{x}{e^b}\right) \\\notag&\leq \max\{e^{\kappa b_n }, e^{-\kappa b_n}\} V(x)w^2+ \max\{e^{-\kappa b }, e^{-\kappa b}\} V(x)w^2\\\notag&\leq C V(x)w^2\in L^1(\mathbb{R}^2).
\end{align}
%where $C_1=\max\{\max\{e^{\kappa b_n }, e^{-\kappa b_n}\},\max\{e^{\kappa b }, e^{-\kappa b}\}\}$.
Thus, by Lebesgue's  dominated convergence Theorem, we deduce
\begin{align}\label{a5}
	\lim_{n\to \infty}\int_{\mathbb{R}^2}	w^2\left[V\left(\frac{x}{e^{b_n}} \right)-V\left(\frac{x}{e^{b}} \right) \right] \ dx=0.
\end{align}
It follows from \eqref{a2}-\eqref{a5} that \eqref{a1} holds, then $\Phi\in C(\tilde E,\mathbb{R}).$ 

\vspace{.2cm}
Next, we prove that \eqref{2-13} holds. It suffices to prove
\begin{align}\label{a6}
	\partial _b\int_{\mathbb{R}^2} V\left(\frac{x}{e^b} \right) w^2dx=-\int_{\mathbb{R}^2}\left[ \nabla V \left(\frac{x}{e^b} \right)\cdot \frac{x}{e^b} \right] w^2dx.
\end{align}
For $|\theta|<1$, it follows from the Newton-Leibniz formula and  Fubini's Theorem that
\begin{align}\label{a7}
	\frac{1}{\theta}\int_{\mathbb{R}^2} \left[ V\left(\frac{x}{e^{b+\theta}} \right) -V\left( \frac{x}{e^b}\right) \right] w^2dx\notag&=-\int_{\mathbb{R}^2}\int_0^1\left( \nabla V\left(\frac{x}{e^{b+\theta t}} \right)\cdot \frac{x}{e^{b+\theta t}} \right)w^2dtdx \\&=-\int_0^1\int_{\mathbb{R}^2}\left( \nabla V\left(\frac{x}{e^{b+\theta t}} \right)\cdot \frac{x}{e^{b+\theta t}} \right)w^2dxdt.
\end{align}
Using \eqref{3.2}, we deduce
\begin{align}
	\left|\int_{\mathbb{R}^2}\left( \nabla V\left(\frac{x}{e^{b+\theta t}} \right)\cdot \frac{x}{e^{b+\theta t}} \right)w^2dx\right|\notag&\leq\int_{\mathbb{R}^2}\left|\left( \nabla V\left(\frac{x}{e^{b+\theta t}} \right)\cdot \frac{x}{e^{b+\theta t}} \right)w^2\right|dx\\\notag&\leq\kappa\int_{\mathbb{R}^2}V\left(\frac{x}{e^{b+\theta t}} \right)w^2dx\\
    \notag&\leq\kappa\max\{e^{\kappa b+\kappa \theta t}, e^{-\kappa b-\kappa\theta t }\} \int_{\mathbb{R}^2} V(x)w^2dx\\
    \notag&\leq C\max\{e^{\kappa b+\kappa t}, e^{-\kappa b+\kappa t }\} 
    \in L^1{((0,1))}.
\end{align}
Then, by Lebesgue's  dominated convergence Theorem, we obtain
\begin{align}\label{a8}
\lim_{\theta\rightarrow 0}\int_0^1\int_{\mathbb{R}^2}\left( \nabla V\left(\frac{x}{e^{b+\theta t}} \right)\cdot \frac{x}{e^{b+\theta t}} \right)w^2dxdt=\int_0^1\lim_{\theta\rightarrow 0}\int_{\mathbb{R}^2}\left( \nabla V\left(\frac{x}{e^{b+\theta t}} \right)\cdot \frac{x}{e^{b+\theta t}} \right)w^2dxdt.
\end{align}
For fixed $t\in (0,1)$, it follows $(V_2)$ and \eqref{3.2} that
\begin{align}\label{a9}
	&\left|\ \left( \nabla V\left(\frac{x}{e^{b+\theta t}} \right)\cdot \frac{x}{e^{b+\theta t}} \right)w^2\ \right|
   % \notag&\leq\left|\ \left( \nabla V\left(\frac{x}{e^{b+\theta t}} \right)\cdot \frac{x}{e^{b+\theta t}} \right)\ \right|w^2\\
    \notag\leq \kappa V\left( \dfrac{x}{e^{b+\theta t}}\right)w^ 2\\
    \notag\leq &\kappa\max\{e^{\kappa b+ \kappa \theta t},e^{-\kappa b-\kappa\theta t}\} V(x)w^2
    \leq CV(x)w^2
    \in L^1(\mathbb{R}^2).
\end{align}
Then, using Lebesgue's dominated convergence  Theorem once again, we have 
 \begin{align}
 	\lim_{\theta\rightarrow 0}\int_{\mathbb{R}^2}\left( \nabla V\left(\frac{x}{e^{b+\theta t}} \right)\cdot \frac{x}{e^{b+\theta t}} \right)w^2dx
    %\notag&=\int_{\mathbb{R}^2}\lim_{\theta\rightarrow 0}\left( \nabla V\left(\frac{x}{e^{b+\theta t}} \right)\cdot \frac{x}{e^{b+\theta t}} \right)w^2dxdt\\
    =\int_{\mathbb{R}^2}\left[\nabla V\left( \dfrac{x}{e^b}\right)\cdot \dfrac{x}{e^b}  \right]w^2 dx.
 \end{align}
Thus, combining \eqref{a7}-\eqref{a9}, we deduce
\begin{align}
	\partial_b\int_{\mathbb{R}^2} V\left( \dfrac{x}{e^b}\right) w^2dx\notag&=\lim_{\theta\rightarrow 0}\frac{1}{\theta}\int_{\mathbb{R}^2} \left[ V\left(\frac{x}{e^{b+\theta}} \right)- V\left( \frac{x}{e^b}\right)  \right]w^2 dx \\\notag&=-\lim_{\theta\rightarrow 0}\int_0^1\int_{\mathbb{R}^2}\left[ \nabla V\left(\dfrac{x}{e^{b+\theta t}} \right) \cdot \frac{x}{e^{b+\theta t}}\right] w^2dxdt \\\notag&=-\int_0^1\lim_{\theta\rightarrow 0}\int_{\mathbb{R}^2}\left[ \nabla V\left(\dfrac{x}{e^{b+\theta t}} \right) \cdot \frac{x}{e^{b+\theta t}}\right] w^2dxdt\\\notag&=-\int_0^1\int_{\mathbb{R}^2}\left[ \nabla V\left( \frac{x}{e^b}\right)\cdot\frac{x}{e^b} \right]w^2dx\\\notag&= -\int_{\mathbb{R}^2}\left[ \nabla V\left( \frac{x}{e^b}\right)\cdot\frac{x}{e^b} \right]w^2dx,
\end{align}
which means that \eqref{a6} holds.
Similar to \eqref{3.7}, we derive $$\int_{\mathbb{R}^2}\left[ \nabla V\left( \frac{x}{e^b}\right)\cdot \dfrac{x}{e^b} \right]w^2dx\in C(\tilde E,\mathbb{R}). $$
Thus, $\Phi_b(b,w)\in C(\tilde E, \mathbb{R})$.

\vspace{.2cm}
In the following, we prove \eqref{2-14} and $\Phi_w(b,w)\in C(\tilde E,E')$. Since the  map $w\mapsto \eta (b,w)$ is linear for fixed $b\in \mathbb{R}$, we obtain 
$$\langle \Phi_w (b,w),v\rangle =\langle I'(\eta(b,w)), \eta(b,v)\rangle , \quad \ b\in \mathbb{R} \ \mathrm{and}\ w,v\in E.$$
Using \eqref{a}, we obtain
\begin{align}
	\langle \Phi_w(b,w),v\rangle \notag&=e^{2ab}\int_{\mathbb{R}^2}\nabla v\nabla wdx+e^{2b(a-1)}\int_{\mathbb{R}^2}V\left(\frac{x}{e^b} \right)vwdx-2abe^{2b(a-1)}\int_{\mathbb{R}^2}vwdx \\\notag&\quad +\lambda e^{4b(a-1)}\int_{\mathbb{R}^2}\int_{\mathbb{R}^2}\ln |x-y|w^2(y)w(x)v(x)dxdy\\\notag&\quad-\lambda be^{4b(a-1)}\int_{\mathbb{R}^2}w^2(y)dy\int_{\mathbb{R}^2}v(x)w(x)dx-e^{2b(a-1)}\int_{\mathbb{R}^2}vw\ln w^2dx.
\end{align}

It remains to prove $\Phi_w(b,w)\in C(\tilde E,E')$. We only need to prove $\int_{\mathbb{R}^2}V\left(\frac{x}{e^b} \right)vwdx\in C(\tilde E,E') $. 
Assume $(b_n,w_n)\rightarrow (b,w)$ in $\tilde E$ as \( n\rightarrow\infty \). It suffices to show that
\begin{align}\label{3.16}
    \int_{\mathbb{R}^2} V\left(\dfrac{x}{e^{b_n}}\right) w_nvdx-  \int_{\mathbb{R}^2} V\left(\dfrac{x}{e^{b}}\right)wvdx=o_n(1)||v||_E.
\end{align}
Note that 
\begin{align}\label{3.16.1}
	\notag&\left| \int_{\mathbb{R}^2} V\left(\frac{x}{e^{b_n}} \right)w_nv dx- \int_{\mathbb{R}^2} V\left(\frac{x}{e^{b}} \right)wv dx \ \right|\\
    \leq&\left|\int_{\mathbb{R}^2} V\left( \dfrac{x}{e^{b_n}}\right) \left[w_n-w \right]vdx\ \right|
    +\left|\int_{\mathbb{R}^2}\left[ V\left(\dfrac{x}{e^{b_n}} \right)-V\left(\frac{x}{e^b} \right)  \right] wvdx\ \right|.
\end{align}
Moreover, using \eqref{3.2} and the H\"{o}lder inequality, we deduce
\begin{align}\label{a14}
	\left|\int_{\mathbb{R}^2} V\left( \dfrac{x}{e^{b_n}}\right) \left[w_n-w \right]vdx \right|
    \notag&\leq \max\{e^{\kappa b_n, -\kappa b_n}\}\int_{\mathbb{R}^2} V(x)|\left[w_n-w \right] v|dx\\
    \notag&\leq C\left( \int_{\mathbb{R}^2} V(x)\left[ w_n-w\right] ^2dx\right) ^{\frac{1}{2}}\left( \int_{\mathbb{R}^2} V(x)v ^2dx\right)^{\frac{1}{2}} \\
    &\leq C||w_n-w||_E||v||_E=o_n(1)||v||_E.
\end{align}
Analogous to \eqref{a5}, we have 
\begin{align*}
    \lim_{n\to \infty}\int_{\mathbb{R}^2}	w^2\left|V\left(\frac{x}{e^{b_n}} \right)-V\left(\frac{x}{e^{b}} \right) \right| \ dx=0,
\end{align*}
which, together with \eqref{3.2} and the H\"{o}lder inequality, gives 
\begin{align}\label{3.17}
    &\left|\int_{\mathbb{R}^2}\left[ V\left(\dfrac{x}{e^{b_n}} \right)-V\left(\frac{x}{e^b} \right)  \right] wvdx\right|\nonumber\\
    \leq &\left[\int_{\mathbb{R}^2}	w^2\left|V\left(\frac{x}{e^{b_n}} \right)-V\left(\frac{x}{e^{b}} \right) \right| \ dx\right]^{\frac{1}{2}} \left[\int_{\mathbb{R}^2}	v^2\left|V\left(\frac{x}{e^{b_n}} \right)-V\left(\frac{x}{e^{b}} \right) \right| \ dx\right]^{\frac{1}{2}}\nonumber\\
    \leq &o_n(1)\left(\max\{e^{\kappa b_n,-\kappa b_n}\}+\max\{e^{\kappa b,-\kappa b}\}\right)^{\frac{1}{2}}\left[\int_{\mathbb{R}^2}	v^2V(x) \ dx\right]^{\frac{1}{2}}\nonumber\\
    =&o_n(1)||v||_E.
\end{align}
Combining \eqref{3.16.1}-\eqref{3.17}, we get \eqref{3.16}.
%It follows from Lemma \ref{lem1} that
%\begin{align}
%	\left|\ \left[ V\left(\dfrac{x}{e^{b_n}} \right)-V\left(\frac{x}{e^b} \right)  \right] wv\ \right|
%    \notag&\leq V\left(\frac{x}{e^{b_n}} \right)wv+ V\left(\frac{x}{e^{b}} \right)wv\\
%    \notag&\leq \max\{e^{\kappa b_n,-\kappa b_n}\} V(x )wv+\max\{e^{\kappa b,-\kappa b}\} V(x )wv\in L^1({E}),
%\end{align}
%then  using  Lebesgue's dominated convergence Theorem yields
%\begin{align}\label{a15}
%	\lim_{n\rightarrow\infty}\int_{\mathbb{R}^2}\left[ V\left(\dfrac{x}{e^{b_n}} \right)-V\left(\frac{x}{e^b} \right)  \right] wvdx=\int_{\mathbb{R}^2}\lim_{n\rightarrow\infty}\left[ V\left(\dfrac{x}{e^{b_n}} \right)-V\left(\frac{x}{e^b} \right)  \right] wvdx=0.
%\end{align}
%Then, Combining \eqref{a14} with \eqref{a15}, we obtain
%\begin{align}
%\notag \lim_{n\rightarrow\infty}	\int_{\mathbb{R}^2} V\left(\dfrac{x}{e^{b_n}}\right) w_nvdx= \int_{\mathbb{R}^2} V\left(\dfrac{x}{e^{b}}\right)wvdx,
%\end{align}
%which means 
So, $\Phi_w(b,w)\in C(\tilde E, E')$. We complete the proof.\hfill$\square$

\vspace{.2cm}
Now, we are ready to establish the existence of a Cerami sequence of the functional $I$ with an additional property.% and investigate $K(u_n)\rightarrow0$ as $n\rightarrow \infty$ in what follows.

\begin{lemma}\label{lemma2.9}
There exists $\{u_n\}\subset E$ such that
\begin{align}
\notag	I(u_n)\rightarrow c_{*,1}, \quad ||I'(u_n)||_{E'}(1+||u_n||_E)\rightarrow0,\quad \mathrm{and }\quad  K(u_n)\rightarrow0
\end{align}
as $n\rightarrow\infty,$ where 
\begin{align*}%\label{2-10}
	K(u)\notag&:=a\int_{\mathbb{R}^2} |\nabla u|^2dx-\int_{\mathbb{R}^2}u^2dx+(a-1)\int_{\mathbb{R}^2} V(x)u^2dx-\dfrac{1}{2}\int_{\mathbb{R}^2} (\nabla V\cdot x)u^2dx\\\notag&\quad +\lambda(a-1)\int_{\mathbb{R}^2}\int_{\mathbb{R}^2} \ln |x-y|u^2(x)u^2(y)dxdy\\&\quad-\dfrac{\lambda}{4}\left(\int_{\mathbb{R}^2} u^2dx\right) ^2-(a-1)\int_{\mathbb{R}^2} u^2\ln u^2dx.
\end{align*}		
\end{lemma}			
\noindent   {\it Proof.}
Denote 
\begin{align*}
    \bar{\Gamma}=\left\lbrace \bar{\gamma}\in C\left( [0,1],\tilde E \right) | \  \bar{\gamma}(0)=(0,0), \Phi(\bar{\gamma}(1))<0\right\rbrace.
\end{align*}
Since $\Gamma$ defined by \eqref{3.4} is not empty and $\{(0,\gamma)\ |\ \gamma\in \Gamma\}\subset  \bar{\Gamma}$, we obtain $\bar{\Gamma}\not=\emptyset$.
 So, we can define the mountain-pass value $$c_0:=\inf_{\bar{\gamma}\in \bar{\Gamma}}\max_{t\in[0,1]}\Phi(\bar{\gamma}(t)).$$
%where $\bar{\Gamma}=\left\lbrace \bar{\gamma}\in C\left( [0,1],\tilde E \right) | \  \bar{\gamma}(0)=(0,0), \Phi(\bar{\gamma}(1))<0\right\rbrace.$
Notice that $\Gamma=\left\lbrace \eta\circ \bar{\gamma}\ |\  \bar{\gamma}\in \bar{\Gamma}\right\rbrace $, thus $c_0=c_{*,1}\geq m_\alpha$. 

\vspace{.2cm}
Define $N=[0,1]$, $N_0=\{0,1\}$ and 
\begin{align*}
    \bar{\Gamma}_0=\left\lbrace \bar{\gamma}\in C\left( N_0,\tilde E \right) | \  \bar{\gamma}(0)=(0,0), \Phi(\bar{\gamma}(1))<0\right\rbrace.
\end{align*}
Then, 
\begin{align*}
    \sup_{\bar{\gamma}\in \bar{\Gamma}_0}\max_{t\in N_0}\Phi(\bar{\gamma}(t))
    =0<c_0.
\end{align*}
%We now apply Lemma \ref{lemma2.5} to the functional $\Phi$, taking    and replacing $E, \Gamma$ with $\tilde E, \bar{\Gamma}$.
%To be more precise, f
For any fixed $n\in \mathbb{N}$, by the definition of $c_{*,1}$, there exists $\gamma_n\in \Gamma$ such that
$$\max_{t\in[0,1]}I(\gamma_n(t))\leq c_{*,1}+\dfrac{1}{n^2}.$$
		Next, we define $\bar{\gamma}_n\in \bar{\Gamma}$	by setting $\bar{\gamma}_n(t):=(0,\gamma_n(t))$, and we obtain
\begin{align*}%\label{2.17}
	\max_{t\in [0,1]}\Phi(\bar{\gamma}_n(t))=\max_{t\in [0,1]}\Phi((0, \gamma_n(t)))=\max_{t\in [0,1]}I(\gamma_n(t))\leq c_{*,1}+\frac{1}{n^2}.
\end{align*}
%By the Lemma \ref{lemma2.6}, there exists  sufficiently small $\alpha>0$ such that 
%\begin{align}\label{2.15}
%	\Phi((b,w))=I(\eta(b,w))\geq m_{\tau}>0, 
%\end{align}
%	furthermore, using Lemma \ref{lemma2.7}, we obtain
%	\begin{align}\label{2.16}
%		\Phi(\bar{\gamma}_n(t))=I(\gamma_n(t))<0,
%	\end{align}
%	for sufficiently large	$t>0$.
Then, it follows from Lemma \ref{lemma2.5} that there exists $(b_n,w_n)\in \tilde E$ such that 
\begin{align}\label{2.18}
	\Phi(b_n, w_n)\rightarrow c_{*,1},
\end{align}
		\begin{align}\label{2.19}
			||\Phi'(b_n,w_n)||_{\tilde{E}'}\left( 1+||(b_n,w_n)||_{\tilde E}\right)\rightarrow 0, 
		\end{align}		
\begin{align*}
	dist\left( (b_n, w_n),(0,\gamma_n([0,1]))\right)\rightarrow 0,
\end{align*}		
which means $b_n\rightarrow 0$ as $n\rightarrow\infty.$		

\vspace{.2cm}
A direct calculation yields 
\begin{align}\label{3.25}
    \Phi_b(b,w)=K(\eta(b, w))\ \ \text{for}\ \ (b, w)\in \tilde E.
\end{align}
For every $(h,v)\in \tilde E$, by Lemma \ref{lemma2.15} and \eqref{3.25}, we deduce
\begin{align}\label{2-21}
	\Phi'(b_n, w_n)(h,v)%&=\left[ \Phi_{b_n}(b_n,w_n), \Phi_{w_n}(b_n,w_n)\right] (h,v)\\
    %\notag&=\Phi_{b_n}(b_n,w_n) h+\Phi_{w_n}(b_n,w_n)v\\
    =\left(K(\eta(b_n, w_n))h,\langle I'(\eta(b_n, w_n)),\eta(b_n, v)\rangle\right).
\end{align}		
Letting $h=1$ and $v=0$ in \eqref{2-21}, we get 
\begin{align}\label{r2}
	K(\eta(b_n,w_n))\rightarrow 0, \quad \mathrm{as}\quad n\rightarrow\infty.
\end{align}
Define $u_n:=\eta(b_n, w_n)$. Combining \eqref{2.18} with \eqref{r2}, we obtain
$$I(u_n)\rightarrow c_{*,1}\quad \mathrm{and }\quad  K(u_n)\rightarrow0 \quad  \mathrm{as }\  n\rightarrow\infty. $$

To conclude, for any given $w\in E$, we consider $v_n:=e^{-ab_n}w(e^{-b_n}\cdot)\in E$. %Set  $h=0$ in \eqref{2-21}, we then derive that
%$$\Phi'(b_n,w_n) (0,w_n)=I'(\eta(b_n,w_n))\eta(b_n,v).$$		
Notice that % We deduce from \eqref{2.19} that 
\begin{align}\label{2-23}
	\left( 1+||u_n||_E\right)|\langle I'(u_n),w\rangle |\notag
    &=\left( 1+||u_n||_E\right)|\langle I'(u_n),\eta(b_n, v_n)\rangle |\\
    \notag&=\left( 1+||u_n||_E\right)|\Phi'(b_n,w_n)(0,v_n)|\\
     &\leq \left( 1+||u_n||_E\right)||\Phi'(b_n,w_n)||_{\tilde E'}||v_n||_{ E}.
   % &\leq\left( 1+||(b_n, w_n)||_{\tilde E}\right)||\Phi'(b_n,w_n)||_{\tilde E'}||v_n||_{ E}\nonumber\\
   % &=o_n(1)||v_n||_{ E}. 
\end{align}
Next, we claim $||v_n||_{E}=(1+o_n(1))||w||_E$ with $o_n(1)\to 0$ uniformly in $w\in E$.
If $b_n\geq 0$, using \eqref{3.2}, we obtain
\begin{align*}%\label{b1}
	|V(e^{b_n}x)- V(x)|\notag&=\left|\int_0^{b_n}\nabla V(e^{s}x)\cdot \left( e^{s}x\right)  ds\right|\leq \kappa \int_0^{b_n} V(e^{s}x)ds\\
    & \leq \kappa \int_0 ^{b_n} e^{s \kappa} V(x)ds 
    %=V(x) \cdot e^{s\kappa}|_0^{b_n}
    =\left(e^{\kappa b_n}-1 \right) V(x)=O(|b_n|) V(x).
\end{align*}
Similarly, if $b_n\leq 0$, we have 
\begin{align*}%\label{b2}
		|V(e^{b_n}x)- V(x)|%\notag&=\left|\int^0_{b_n}\nabla V(e^{s}x)\cdot \left( e^{s}x\right)  ds\right|\leq \kappa \int^0_{b_n} V(e^{s}x)ds\\& \leq \kappa \int^0 _{b_n} e^{-s \kappa} V(x)ds =-V(x) \cdot e^{-s\kappa}|^0_{b_n}=\left(e^{-\kappa b_n}-1 \right) V(x)
        =O(|b_n|) V(x).
\end{align*}
%then, using \eqref{b1}-\eqref{b2} yields
Therefore, 
\begin{align}\label{b3}
	\left|\int_{\mathbb{R}^2} \left(V(e^{b_n}x) -V(x)\right) w^2(x)dx\right|\leq O(|b_n|)\int_{\mathbb{R}^2} V(x)w^2(x)dx =o_n(1)||w||_E^2.
\end{align}
%Furthermore,
%\begin{align}\label{b4}
%	|\ln (1+e^{2b_n}|x|)-\ln %(1+|x|)|=\left|\int_0^{b_n}\dfrac{2e^{2s}|x|}{1+e^{2s}|x|}\right|\leq \left|\int_0^{b_n}2ds\right|=2|b_n|,
%\end{align}
%then, it follows  that
 %\begin{align}\label{b5}
% \left|	\int_{\mathbb{R}^2} \left[\ln (1+e^{2b_n}|x|)-\ln (1+|x|) \right]w^2dx\right|\leq 2|b_n|\int_{\mathbb{R}^2} w^2dx =o(1)\int_{\mathbb{R}^2} w^2dx=o(1)||w||_E 
% \end{align}
Using $b_n\rightarrow 0$ and  \eqref{b3}, we deduce
\begin{align}\label{2-24}
	||v_n||_{E}^2\notag&=\int_{\mathbb{R}^2}\left[|\nabla v_n|^2+V(x)v_n^2 \right]dx \\\notag&=e^{-4b_n}\int_{\mathbb{R}^2}|\nabla w|^2dx+e^{-2b_n}\int_{\mathbb{R}^2}V\left(e^{b_n} x\right)w^2 dx\\
    \notag&=\left(1+o_n(1) \right)\int_{\mathbb{R}^2} |\nabla w|^2dx+\left(1+o_n(1) \right) \left(\int_{\mathbb{R}^2}  V(x)w^2 dx+o_n(1)||w||_E^2\right)\\
    &=(1+o_n(1))||w||_E^2.
\end{align}
Similarly, we have 
\begin{align}\label{3.30}
    ||u_n||_{E}=(1+o_n(1))||w_n||_E.
\end{align}
Combining \eqref{2.19}, \eqref{2-23}, \eqref{2-24} and \eqref{3.30}, we obtain
\begin{align*}
    \left( 1+||u_n||_E\right)|\langle I'(u_n),w\rangle |
\leq &C\left( 1+||w_n||_E\right)||\Phi'(b_n,w_n)||_{\tilde E'}||w||_E\\
\leq &C\left( 1+||(b_n, w_n)||_{\tilde E}\right)||\Phi'(b_n,w_n)||_{\tilde E'}||w||_E
=o_n(1)||w||_E,%\quad \mathrm{as}\quad n\rightarrow\infty,
\end{align*}
which implies 
\begin{align*}
    ||I'(u_n)||_{E'}(1+||u_n||_E)\rightarrow0.
\end{align*}
This completes the proof.\hfill$\square$

\vspace{.2cm}
The next lemma will be used to establish the boundedness of the Cerami sequence obtained in Lemma \ref{lemma2.9}.
\begin{lemma}\label{lemma2.10}
	Suppose that $\{u_n\}\subset E$ satisfies 
\begin{align}\label{1}
	c=\sup_{n\in \mathbb{N}} I(u_n)<+\infty, \quad ||I'(u_n)||_{E'}(1+||u_n||_E)\rightarrow0,\quad \mathrm{and }\quad  K(u_n)\rightarrow0
\end{align}
as $n\rightarrow \infty$. Then $\{u_n\}$ is bounded in $E$.
%	$$\int_{\mathbb{R}^2} \left( |\nabla u_n|^2dx+V(x)u^2_n\right) dx\leq C.$$
\end{lemma}		
\noindent   {\it Proof.} First, using Lemma \ref{lemmaa}, we obtain
\begin{align}\label{2-25}
c+o_n(1)\notag&
\geq I(u_n)-\frac{1}{4(a-1)}K(u_n)\\
\notag&=\frac{a-2}{4(a-1)}\int_{\mathbb{R}^2}|\nabla u_n|^2dx+\frac{2a-1}{4(a-1)}\int_{\mathbb{R}^2} u_n^2dx+\frac{1}{4}\int_{\mathbb{R}^2}V(x)u_n^2dx\\
\notag&\quad+\frac{1}{8(a-1)}\int_{\mathbb{R}^2}(\nabla V\cdot x)u_n^2dx +\dfrac{\lambda}{16(a-1)}\left( \int_{\mathbb{R}^2} u_n^2dx\right)^2-\dfrac{1}{4}\int_{\mathbb{R}^2} u_n^2\ln u_n^2dx\\
\notag&\geq\frac{a-2}{4(a-1)}\int_{\mathbb{R}^2}|\nabla u_n|^2dx+ \frac{2a-1}{4(a-1)}\int_{\mathbb{R}^2} u_n^2dx+\frac{2a-2-\kappa}{8(a-1)}\int_{\mathbb{R}^2}V(x)u_n^2dx\\&\quad +\dfrac{\lambda}{16(a-1)}\left( \int_{\mathbb{R}^2} u_n^2dx\right)^2-C\int_{\mathbb{R}^2} |u_n|^qdx,
\end{align}
where $2<q<3.$

In the following, we claim
\begin{align}\label{2-26}
\int_{\mathbb{R}^2}|\nabla u_n|^2dx\leq C, \quad \mathrm{for} \quad n\in \mathbb{N}.
\end{align} 
Suppose, for contradiction, that the claim fails. Upon passing to a subsequence if necessary, we may assume that 
\begin{align}
	\notag\int_{\mathbb{R}^2}|\nabla u_n|^2dx\rightarrow \infty, \quad \mathrm{as} \quad n\rightarrow \infty. 
\end{align}
We set $s_n:=||\nabla u_n||^{{-\frac{1}{2}}}_2$ for $n\in \mathbb{N}$, which means $s_n\rightarrow 0$ as $n\rightarrow\infty$. %We then introduce the rescaled functions  $w_n\in E$ via $w_n(x)=s_n^2 u_n(s_n x) $. 
Define $w_n(x):=s_n^2 u_n(s_n x) \in E$.
It follows that
\begin{align}\label{2-28}
	\int_{\mathbb{R}^2}|\nabla w_n|^2dx=s_n^4\int_{\mathbb{R}^2} |\nabla u_n|^2dx=1
\end{align}
and
\begin{align}\label{2-29}
\int_{\mathbb{R}^2}|w_n|^qdx=s_n^{2q}\int_{\mathbb{R}^2}\left| u_n(s_nx)\right| ^qdx=s_n^{2q-2}\int_{\mathbb{R}^2}|u_n|^qdx.
\end{align}
Using the Gagliardo-Nirenberg inequality, we deduce
\begin{align}\label{3.35}
\int_{\mathbb{R}^2}|w_n|^qdx\leq C \int_{\mathbb{R}^2}w_n^2dx\left( \int_{\mathbb{R}^2}|\nabla w_n|^2dx\right) ^{\frac{q-2}{2}}=C\int_{\mathbb{R}^2}w_n^2dx.
\end{align}
Multiplying both sides of \eqref{2-25} by  $s_n^4$ and invoking \eqref{2-29}, we derive
\begin{align}\label{3.37.1}
c	{s_n^4}+o\left( s_n^4\right) \notag&
\geq\frac{(a-2)s_n^4}{4(a-1)}\int_{\mathbb{R}^2}|\nabla u_n|^2dx+\dfrac{(2a-1)s_n^4}{4(a-1)}\int_{\mathbb{R}^2}u_n^2dx-{Cs_n^{4}}\int_{\mathbb{R}^2}|u_n|^qdx\\
\notag&\quad +\dfrac{(2a-2-\kappa)s^4_n}{8(a-1)}\int_{\mathbb{R}^2}V\left(x \right)u_n^2dx+\frac{\lambda s_n^4}{16(a-1)}\left( \int_{\mathbb{R}^2} u_n^2dx\right)^2 \\
\notag&\geq\frac{a-2}{4(a-1)}+\dfrac{(2a-1)s_n^2}{4(a-1)} \int_{\mathbb{R}^2} w_n^2dx+\dfrac{(2a-2-\kappa)s_n^2}{8(a-1)}\int_{\mathbb{R}^2} V\left( s_n x\right)w_n^2dx \\
\notag&\quad +\dfrac{\lambda}{16(a-1)}\left(\int_{\mathbb{R}^2} w_n^2dx \right)^2-{C s_n^{6-2q}}\int_{\mathbb{R}^2} w_n^2dx\\
%\notag&\geq\dfrac{3s_n^2}{4} \int_{\mathbb{R}^2} w_n^2dx +\dfrac{\left(2-\kappa\right)  s_n^2}{8}\int_{\mathbb{R}^2} V\left(s_n x  \right)w_n^2dx\\
%    \notag&\quad +\dfrac{\lambda}{16}\left(\int_{\mathbb{R}^2} w_n^2dx \right)^2-{C s_n^{6-2q}}\int_{\mathbb{R}^2} w_n^2dx	\\
&\geq %C s_n^2\int_{\mathbb{R}^2}w_n^2dx+
    \dfrac{\lambda}{16(a-1)}\left( \int_{\mathbb{R}^2} w_n^2dx\right)^2-{Cs_n^{6-2q}}\int_{\mathbb{R}^2} w_n^2dx
\end{align}
which implies 
\begin{align}
	\int_{\mathbb{R}^2} w_n^2dx=&O\left(s_n^{6-2q} \right),\label{2-32}\\
    s_n^2\int_{\mathbb{R}^2} V\left( s_n x\right)w_n^2dx
    \leq &Cs_n^4+{Cs_n^{6-2q}}\int_{\mathbb{R}^2} w_n^2dx
    =O\left(s_n^{12-4q} \right).\label{3.37}
   %s_n^2\int_{\mathbb{R}^2}\left(\nabla V\left(s_n x \right) \cdot s_nx \right)w_n^2dx
  % \leq &Cs_n^4+{Cs_n^{6-2q}}\int_{\mathbb{R}^2} w_n^2dx
  %  =O\left(s_n^{12-4q} \right)
\end{align}
%where we use $(V_3)$.
It follows from \eqref{3.37.1}-\eqref{3.37} that 
\begin{align*}
   C s_n^4\geq \frac{a-2}{4(a-1)}+o_n(1),
\end{align*}
which is a contradiction. 
Therefore, \eqref{2-26} holds.

\vspace{.1cm}
It follows from the Gagliardo-Nirenberg inequality and \eqref{2-26} that 
\begin{align*}
    \int_{\mathbb{R}^2}|u_n|^qdx\leq C \int_{\mathbb{R}^2}u_n^2dx\left( \int_{\mathbb{R}^2}|\nabla u_n|^2dx\right) ^{\frac{q-2}{2}}\leq C\int_{\mathbb{R}^2}u_n^2dx,
\end{align*}
which, together with \eqref{2-25}, yields 
\begin{align*}
    c+o_n(1)
    \geq &\frac{a-2}{4(a-1)}\int_{\mathbb{R}^2}|\nabla u_n|^2dx+ \frac{2a-1}{4(a-1)}\int_{\mathbb{R}^2} u_n^2dx+\frac{2a-2-\kappa}{8(a-1)}\int_{\mathbb{R}^2}V(x)u_n^2dx\\
    &+\dfrac{\lambda}{16(a-1)}\left( \int_{\mathbb{R}^2} u_n^2dx\right)^2-C\int_{\mathbb{R}^2} u_n^2dx.
\end{align*}
Hence, 
\begin{align*}
    \int_{\mathbb{R}^2} u_n^2dx\leq C\ \ \text{and}\ \ \int_{\mathbb{R}^2}V(x)u_n^2dx\leq C.
\end{align*}
So, $\{u_n\}$ is bounded in $E$. This completes the proof.
	%Moreover, it follows from \eqref{1} that
%	\begin{align}
	%\notag	\langle I'(u_n),u_n\rangle -K(u_n)=-\int_{\mathbb{R}^2} |\nabla u_n|^2dx+\int_{\mathbb{R}^2}u^2dx+\frac{\lambda}{4}\left(\int_{\mathbb{R}^2} u_n^2 dx\right)^2+\dfrac{1}{2}\int_{\mathbb{R}^2} \left( \nabla V\cdot x\right) u_n^2dx,
%	\end{align}
%	which means
	%\begin{align}
	%\notag	\frac{\lambda}{4}\left(\int_{\mathbb{R}^2} u_n^2 dx\right)^2+\int_{\mathbb{R}^2} u^2dx+\dfrac{1}{2}\int_{\mathbb{R}^2} \left( \nabla V\cdot x\right) u_n^2dx=\langle I'(u_n),u_n\rangle -K(u_n)+\int_{\mathbb{R}^2}|\nabla u_n|^2dx\leq o(1)+C_1,
%	\end{align}	
	%	then, $u_n$ is bounded in $L^2(\mathbb{R}^2)$, then $(u_n)$ is bounded in $H^1(\mathbb{R}^2)$.
%Moreover, using 	 \eqref{1-7}, we deduce
%\begin{align}\label{8}
%	\frac{1}{2}\int_{\mathbb{R}^2} V(x)u_n^2dx &\notag=I(u_n)-\frac{1}{2}\int_{\mathbb{R}^2}\left( |\nabla u_n|^2+u^2_n\right) dx\\\notag&\quad+\frac{1}{2}\int_{\mathbb{R}^2}u_n^2\ln u_n^2dx-\frac{\lambda}{4}\int_{\mathbb{R}^2}\int_{\mathbb{R}^2} \ln |x-y|u_n^2(x)u_n^2(y)dxdy\\\notag&\leq I(u_n)+\dfrac{1}{2}\int_{\mathbb{R}^2}|u_n|^qdx+\frac{\lambda}{4}\int_{\mathbb{R}^2}\int_{\mathbb{R}^2}\ln  \left(1+\dfrac{1}{|x-y|} \right) u_n^2(x)u_n^2(y)dxdx\\&\leq c+C_1+C_2||u_n||_{\frac{8}{3}}^2\leq C.
%\end{align}
%		Combining  \eqref{2-26} with \eqref{8}, $u_n$ is bounded in  $E$. This completes the proof. 
\hfill$\square$

   \vspace{.2cm}
        In the following, we will establish the compactness for the sequence satisfying \eqref{1}.
		
	\begin{lemma}\label{lemma2.11}
		Let $\{u_n\}\subset E$ be a sequence satisfying \eqref{1}, and assume $\liminf_{n\to \infty}||u_n||_E>0$. Then, %one of the following occurs:
		%\\(i) $||u_n||_E\rightarrow 0$ and $I(u_n)\rightarrow 0$ as $n\rightarrow\infty$.
		there exists $u\in E\backslash\{0\}$ such that $u_n\rightarrow u$ in $E$ as $n\rightarrow\infty$,
and $u$ is a solution of equation \eqref{1-4}.
\end{lemma}	
\noindent   {\it Proof.}	By Lemma \ref{lemma2.10}, $\{u_n\}$  is bounded in $E$, so there exists $u\in E$ such that $u_n\rightharpoonup  u$ in $E$. According to Lemma \ref{lemma1}, we have $u_n\rightarrow u$ in $L^s{(\mathbb{R}^2)}\left( s\geq 2\right)$.

In the following, we will prove $u\not\equiv 0$.
Assume by contradiction that $u\equiv 0$, then $u_n\rightarrow 0 $ in $L^s{(\mathbb{R}^2)}\left( s\geq 2\right)$.  Using \eqref{1} and \eqref{2.8.1}, we deduce
\begin{align}
	\notag&\int_{\mathbb{R}^2}\left[ |\nabla u_n|^2+V(x)u_n^2\right] dx+\lambda\int_{\mathbb{R}^2}\int_{\mathbb{R}^2} \ln (1+|x-y|)u_n^2(x)u_n^2(y)dxdy\\
\notag=&\langle I'(u_n),u_n\rangle +\lambda\int_{\mathbb{R}^2}\int_{\mathbb{R}^2}\ln \left(1+\dfrac{1}{|x-y|} \right)u_n^2(x)u_n^2(y)dxdy+\int_{\mathbb{R}^2}u_n^2\ln u_n^2dx\\
\notag\leq &o_n(1) +C||u_n||_{\frac{8}{3}}^4+C||u_n||_q^q\rightarrow0,\quad \mathrm{as}\quad n\rightarrow\infty,
\end{align}
which implies $||u_n||_E\rightarrow 0$ as $n\rightarrow\infty$. This contradicts $\liminf_{n\to \infty}||u_n||_E>0$. Thus, $u\not\equiv 0$.

Next, we prove $u_n\rightarrow u$ in $E$ as $n\rightarrow\infty$.  By \eqref{1}, we deduce
\begin{align}\label{2-38}
	\notag o_n(1)&=\langle I'(u_n),u_n-u\rangle \\
    \notag&=\int_{\mathbb{R}^2}\left[  \nabla u_n\nabla(u_n-u)+ V(x)u_n(u_n-u)\right] dx-\int_{\mathbb{R}^2} u_n(u_n-u)\ln u_n^2dx\\
    \notag&\quad+\lambda\int_{\mathbb{R}^2}\int_{\mathbb{R}^2} \ln|x-y|u_n^2(y)u_n(x)(u_n(x)-u(x))dxdy\\
    %\notag&=||u_n||^2_E-||u||^2_E+\lambda\int_{\mathbb{R}^2}\int_{\mathbb{R}^2} \ln |x-y|u_n^2(y)u_n(x)(u_n(x)-u(x))dxdy+o_n(1)\\
    \notag&=||u_n||^2_E-||u||^2_E+\lambda\int_{\mathbb{R}^2}\int_{\mathbb{R}^2}\ln\left(1+|x-y| \right)u_n^2(y)u_n(x)(u_n(x)-u(x))dxdy\\
    &\quad-\lambda\int_{\mathbb{R}^2}\int_{\mathbb{R}^2}\ln\left( 1+\dfrac{1}{|x-y|}\right)u_n^2(y)u_n(x)(u_n(x)-u(x))dxdy+o_n(1),
\end{align}
where we used
\begin{align}\label{e1}
\left|	\int_{\mathbb{R}^2}u_n(u_n-u)\ln u_n^2dx\right|\notag&\leq C\int_{\mathbb{R}^2} |u_n-u|(|u_n|^{1-\frac{\varepsilon_0}{2}}+|u_n|^{1+\frac{\varepsilon_0}{2}})dx\\\notag&\leq C\left(||u_n-u||_{L^2(\mathbb{R}^2)} ||u_n||_{L^{2-\varepsilon_0}(\mathbb{R}^2)}^{\frac{2-\varepsilon_0}{2}}+||u_n-u||_{L^2(\mathbb{R}^2)}||u_n||_{L^{2+\varepsilon_0}(\mathbb{R}^2)}^{\frac{2+\varepsilon_0}{2}}\right)\\ 
&\leq C||u_n-u||_{L^2(\mathbb{R}^2)} \left(||u_n||_E^{\frac{2-\varepsilon_0}{2}}+||u_n||_E^{\frac{2+\varepsilon_0}{2}}\right) 
=o_n(1).
\end{align}
Here $\varepsilon_0 := \frac{\gamma}{\gamma+2}$.
It follows from \eqref{1-7} that
\begin{align}\label{3.41}
	\notag&\left|\int_{\mathbb{R}^2}\int_{\mathbb{R}^2}\ln \left(1+\dfrac{1}{|x-y|} \right) u_n^2(y)u_n(x)(u_n(x)-u(x))dxdy\right|\\
   % \notag&\leq\int_{\mathbb{R}^2}\int_{\mathbb{R}^2}\dfrac{1}{|x-y|}u_n^2(y)u_n(x)(u_n(x)-u(x))dxdy\\
 \leq &C||u_n||_{\frac{8}{3}}^3||u_n-u||_{\frac{8}{3}}\rightarrow 0, \quad \mathrm{as}\quad n\rightarrow\infty.
\end{align}		
Notice that
\begin{align}\
\notag&\int_{\mathbb{R}^2}\int_{\mathbb{R}^2}\ln \left( 1+|x-y|\right) u_n^2(y)u_n(x)(u_n(x)-u(x))dxdy\\
\notag=&\int_{\mathbb{R}^2}\int_{\mathbb{R}^2} \ln\left( 1+|x-y|\right)u_n^2(y)(u_n(x)-u(x))^2dxdy\\
\notag&+\int_{\mathbb{R}^2}\int_{\mathbb{R}^2} \ln\left( 1+|x-y|\right)u_n^2(y)u(x)(u_n(x)-u(x))dxdy.
\end{align}		
Due to $u_n\rightharpoonup u$ in $E$ as $n\rightarrow \infty$, it follows from Lemma \ref{lemma2.3}  that 
\begin{align}
	\notag\int_{\mathbb{R}^2}\int_{\mathbb{R}^2} \ln\left( 1+|x-y|\right)u_n^2(y)u(x)(u_n(x)-u(x))dxdy\rightarrow0, \quad \mathrm{as}\quad n\rightarrow\infty,
\end{align}
which implies 
\begin{align}\label{2-42}
\notag&	\int_{\mathbb{R}^2}\int_{\mathbb{R}^2}\ln \left( 1+|x-y|\right) u_n^2(y)u_n(x)(u_n(x)-u(x))dxdy\\
=&\int_{\mathbb{R}^2}\int_{\mathbb{R}^2} \ln\left( 1+|x-y|\right)u_n^2(y)(u_n(x)-u(x))^2dxdy+o_n(1).
\end{align}		
Combining \eqref{2-38}, \eqref{3.41} and \eqref{2-42}, we obtain
\begin{align}
	\notag o_n(1)%&=\langle I'(u_n),u_n-u\rangle \\
    \notag
    &=||u_n||^2_E-||u||^2_E+\lambda\int_{\mathbb{R}^2}\int_{\mathbb{R}^2} \ln \left( 1+|x-y|\right)u_n^2(y)(u_n(x)-u(x))^2dxdy+o_n(1)\\
    \notag&\geq||u_n||^2_E-||u||^2_E+o_n(1).
\end{align}		
Consequently, $||u_n||_E\rightarrow ||u||_E$, which, together with $u_n\rightharpoonup u$ in $E$, implies that  $u_n\rightarrow u$ in $E$.
	%	Finally,  we need to prove 	$I'(u)=0$. Let $v\in E$, it follows from \eqref{1} that
        Therefore, 
        \begin{align*}
           ||I'(u)||_{E'} =\lim_{n\to \infty} ||I'(u_n)||_{E'}=0,
        \end{align*}
%\begin{align}
%\notag	|I'(u)v|=\lim_{n\to \infty}|I'(u_n)v|\leq \lim_{n\to \infty}\left[||I'(u_n)||_E||v||_{E'} \right]=0,
%\end{align}
which means that $u$ is a nontrivial solution of equation	\eqref{1-4}.	This completes the proof. \hfill$\square$

        \vspace{.2cm}
	\noindent{\textbf {Proof of Theorem \ref{theorem1.1}.}} 
    Let $\{u_n\}$ be as in Lemma \ref{lemma2.9}. We claim that 
    \begin{align*}
        \delta:=\liminf_{n\to \infty}||u_n||_E>0.
    \end{align*}
    Assume by contradiction that $\delta=0$. Then, passing to a subsequence if necessary, we have $u_n\to 0$ in $E$, which is a contradiction with $I(u_n)\rightarrow c_{*,1}>0$. So, $\delta>0$.
    Using Lemma \ref{lemma2.11}, we obtain that equation \eqref{1-4} has a nontrivial solution with $I(u)=c_{*,1}$.
    %there exists a critical point  $u\in E\backslash \{0\}$ of $I$ such that $I(u)=c_*$, 
   % which means that Theorem \ref{theorem1.1} (i) holds. In the following, we give the proof of Theorem \ref{theorem1.1} (ii). We set 
   Hence,
		$$M:=\{u\in E\backslash\{0\}\ |\ I'(u)=0\}$$
		%By Theorem \ref{theorem1.1} (i), it follows that $M$
        is not an empty set, and we can define
        \begin{align*}
            c_{**}:=\inf_{w\in M}I(w)\in [-\infty, c_{*,1}].
        \end{align*}
        Let $\{u_n\}\subset M$ be a minimizing sequence of $c_{**}$, that is, $I(u_n)\rightarrow c_{**}$ and $I'(u_n)=0$.
		Since $\{u_n\}$ is a sequence of solutions to equation \eqref{1-4}, it satisfies the Pohozaev identity $P(u_n)=0$, where 
        \begin{align*}%\label{r1}
P(u)\notag&:=\int_{\mathbb{R}^2}u^2dx+\int_{\mathbb{R}^2} V(x)u^2dx+\dfrac{1}{2}\int_{\mathbb{R}^2}(\nabla V\cdot x) u^2dx\\&\quad +\lambda\int_{\mathbb{R}^2}\int_{\mathbb{R}^2} \ln|x-y|u^2(x)u^2(y)dxdy+\dfrac{\lambda}{4}\left(\int_{\mathbb{R}^2} u^2dx \right)^2-\int_{\mathbb{R}^2}u^2\ln u^2dx.
\end{align*}
Therefore, $K(u_n)=a\langle I'(u_n),u_n\rangle -P(u_n)=0$.
Due to $I'(u_n)=0$, by Lemma \ref{lemma2.6}, we derive 
\begin{align*}%\label{2-47}
	||u_n||_E\geq\alpha>0.
\end{align*}
%According to the  definition of $M$	and Lemma \ref{lemma2.8}, the sequence $\{u_n\}$ satisfies \eqref{1}. Moreover, we claim 
%If \eqref{2-47} fails to hold, using Lemma \ref{lemma2.6}, we have $\langle I'(u_n),u_n\rangle >0$, which implies a contradiction. Then, \eqref{2-47} holds.
Consequently, it follows from Lemma \ref{lemma2.11} that there exists $v\in E\backslash\{0\}$ such that $u_n\rightarrow v$ in $E$, and $v$ is a nontrivial solution of equation \eqref{1-4}. Then,
\begin{align}
\notag 	c_{**}=\lim_{n\to\infty} I(u_n)=I(v)>-\infty,
\end{align}
which shows that $v$ is a ground state solution of equation \eqref{1-4}.
\hfill$\square$

	\section{Ground state solution for $\lambda<0$}\label{s4}

This section is devoted to the proof of  Theorem \ref{theorem1.3}. First, we
 verify that the energy functional $I$ satisfies the mountain pass geometry when $\lambda<0$.
		\begin{lemma}\label{lemma4.1}
		There exists $\alpha>0$ such that for all $0<\tau\leq \alpha$,
		$$ \bar{m}_{\tau}:=\inf \{I(u)\ |\ u\in E, ||u||_E=\tau\}>0$$
		and
		$$\bar{e}_{\tau}:=\inf \{\langle I'(u),u\rangle \ | \ u\in E, ||u||_E=\tau\}>0. $$
	\end{lemma}				
	\noindent   {\it Proof.} 
	%For $q>2$ and $u\in E$, it follows from Lemmas \ref{lemma1} and \ref{lemmaa} that 
	%\begin{align}\label{4.1}
	%	\int_{\mathbb{R}^2} u^2\ln u^2dx\leq \int_{\{|u| > 1\}} u^2\ln u^2dx \leq C\int_{\{|u| > 1\}} |u|^qdx\leq C||u||_q^q\leq C||u||^q_E,
	%\end{align} 
	%where $C>0$ is independent of $u$. Moreover, 
	For every $u\in E$,	 	using \eqref{c5} and \eqref{2.8.0}, we deduce
	\begin{align}
		I(u)\notag&%=\dfrac{1}{2}\int_{\mathbb{R}^2}\left[ |\nabla u|^2+u^2+V(x)u^2\right]dx -\frac{1}{2}\int_{\mathbb{R}^2}u^2\ln u^2dx\\\notag&\quad +\frac{\lambda}{4}\int_{\mathbb{R}^2}\int_{\mathbb{R}^2} \ln |x-y|u^2(x)u^2(y)dxdy\\\
        =\dfrac{1}{2}\int_{\mathbb{R}^2}\left[ |\nabla u|^2+u^2+V(x)u^2\right]dx -\frac{1}{2}\int_{\mathbb{R}^2}u^2\ln u^2dx\\
        \notag&\quad +\frac{\lambda}{4}\int_{\mathbb{R}^2}\int_{\mathbb{R}^2} \ln \left( 1+|x-y|\right) u^2(x)u^2(y)dxdy-\frac{\lambda}{4}\int_{\mathbb{R}^2}\int_{\mathbb{R}^2} \ln \left( 1+\dfrac{1}{|x-y|}\right) u^2(x)u^2(y)dxdy\\
      %  \notag&\geq\dfrac{1}{2}||u||^2_E-C_1||u||^q_E+\frac{\lambda}{4}\int_{\mathbb{R}^2}\int_{\mathbb{R}^2} \ln \left( 1+{|x-y|}\right) u^2(x)u^2(y)dxdy\\
    %    \notag&\geq\dfrac{1}{2}||u||^2_E-C_1||u||^q_E+\frac{\lambda}{4}||u||_2^2\int_{\mathbb{R}^2}\ \ln \left( 1+{|x|}\right) u^2(x)dx+\frac{\lambda}{4}||u||_2^2\int_{\mathbb{R}^2}\ \ln \left( 1+{|y|}\right) u^2(y)dy\\
        \notag& \geq\dfrac{1}{2}||u||^2_E-C||u||^q_E-C||u||_2^2\int_{\mathbb{R}^2}V(x)u^2dx\\
        \notag&\geq\dfrac{1}{2}||u||^2_E-C||u||^q_E-C||u||_E^4
        =||u||_E^2\left( \frac{1}{2}-C||u||_E^{q-2}-C||u||_E^2\right) , 
	\end{align}
	where $q>2$. This means that there exists $\alpha>0$ such that  $\bar{m}_{\tau}>0$ for $0<\tau\leq\alpha$.  Likewise, we may take $\alpha > 0$ sufficiently small such that $\bar{e}_{\tau} > 0$ for all $0 < \tau \leq \alpha$.
    \qed 
	
	%Moreover,
	%\begin{align}
	%	\langle I'(u),u\rangle \notag&=\int_{\mathbb{R}^2}\left[ |\nabla u|^2+V(x)u^2\right]dx-\int_{\mathbb{R}^2}u^2\ln u^2dx\\\notag&\quad +\lambda\int_{\mathbb{R}^2}\int_{\mathbb{R}^2}\ln |x-y| u^2(x)u^2(y)dxdy\\\notag &\geq\int_{\mathbb{R}^2}\left[ |\nabla u|^2+V(x)u^2\right]dx-\int_{\mathbb{R}^2}u^2\ln u^2dx\\\notag&\quad +\lambda\int_{\mathbb{R}^2}\int_{\mathbb{R}^2}\ln \left( 1+{|x-y|}\right)  u^2(x)u^2(y)dxdy\\\notag&\geq\int_{\mathbb{R}^2} \left[ |\nabla u|^2+V(x)u^2\right] dx-\int_{\mathbb{R}^2}u^qdx+\lambda\int_{\mathbb{R}^2}\int_{\mathbb{R}^2}\ln(1+|y|) u^2(x)u^2(y)dxdy\\\notag&\quad +\lambda\int_{\mathbb{R}^2}\int_{\mathbb{R}^2}\ln(1+|x|) u^2(x)u^2(y)dxdy\\\notag&\geq ||u||_E^2-C_1||u||_E^q+\lambda C_2||u||_E^4=||u||_E^2\left(1-C_1||u||_E^{q-2}+\lambda C_2||u||_E^2 \right),
	%\end{align}
	%where $C_1, C_2>0$ are constants and $2< q<2^*$. This means that $e_{\tau}>0$ for $\alpha>0$ small enough. This completes the proof. \hfill$\square$
\vspace{.2cm}    
	Next, we show that there exists $v\in E$ such that $I(v)<0$. However, due to $\lambda<0$, the function $u_\theta$ defined in Lemma \ref{lemma2.7} is not applicable here. In order to overcome this difficulty, we define %$u_t(x):=t^ru\left(\frac{x}{t} \right)$
    \begin{align}\label{4.1}
        u_t(x):=t^ru\left(\frac{x}{t} \right)
    \end{align}
    with $r>\max\{\frac{\kappa}{2}-1,0\}$ for $u\in E$ and $t>0$. Then, we have the following lemma.
    \begin{lemma}\label{lemma4.2}
		Assume that  $u\in E\setminus \{0\}$, then %for all $x\in \mathbb{R}^2$, there exists $u_t\in E$ such that 
		$$I(u_t)\rightarrow-\infty, \quad \mathrm{as} \ \ t\rightarrow+\infty.$$
	\end{lemma}					
	\noindent   {\it Proof.} 
	First, we claim $u_t \in E $ for $u\in E$. Note that 
	\begin{align}
		\int_{\mathbb{R}^2} |\nabla u_t|^2dx&=t^{2r} \int_{\mathbb{R}^2}|\nabla u|^2dx<+\infty,\label{4.3}\\
        \int_{\mathbb{R}^2} V(x)u_t^2dx
        &=t^{2r+2}\int_{\mathbb{R}^2}V\left( tx\right) u^2dx\leq t^{2r+2}\max \{t^{\kappa},t^{-\kappa}\}\int_{\mathbb{R}^2} V(x)u^2dx<+\infty,\label{4.4}
	\end{align}
	where we used $(V_2)$ and \eqref{3.2}. So, $u_t \in E$.
	For $u\in E\backslash \{0\}$ and $t>1$, it follows from \eqref{4.3} and \eqref{4.4} that 
	\begin{align}
		I(u_t)
        %\notag&=\dfrac{1}{2}\int_{\mathbb{R}^2}\left[ |\nabla u_t|^2+u_t^2+V(x)u_t^2\right]dx -\frac{1}{2}\int_{\mathbb{R}^2}u_t^2\ln u_t^2dx\\
       % \notag&\quad +\frac{\lambda}{4}\int_{\mathbb{R}^2}\int_{\mathbb{R}^2} \ln |x-y|u_t^2(x)u_t^2(y)dxdy\\
       % \notag&=\dfrac{t^{2r}}{2}\int_{\mathbb{R}^2}|\nabla u|^2dx+\dfrac{t^{2r+2}}{2}\int_{\mathbb{R}^2} u^2dx+\dfrac {t^{2r+2}}{2}\int_{\mathbb{R}^2} V\left( tx\right) u^2dx+\dfrac{\lambda t^{4r+4}}{4} \int_{\mathbb{R}^2}\int_{\mathbb{R}^2}\ln |x-y| u^2(x)u^2(y)dxdy\\
%\notag&\quad+\frac{\lambda t^{4r+4}\ln t}{4}\left( \int_{\mathbb{R}^2}u^2dx\right) ^2 -rt^{2r+2}\ln t\int_{\mathbb{R}^2} u^2dx-\frac{t^{2r+2}}{2}\int_{\mathbb{R}^2}u^2\ln u^2dx\\
        \notag\leq&\dfrac{t^{2r}}{2}\int_{\mathbb{R}^2}|\nabla u|^2dx+\dfrac{t^{2r+2}}{2}\int_{\mathbb{R}^2} u^2dx+\dfrac {t^{2r+2+\kappa}}{2}\int_{\mathbb{R}^2} V\left( x\right) u^2dx\\
      \notag&  +\dfrac{\lambda t^{4r+4}}{4} \int_{\mathbb{R}^2}\int_{\mathbb{R}^2}\ln |x-y| u^2(x)u^2(y)dxdy+\frac{\lambda t^{4r+4}\ln t}{4}\left( \int_{\mathbb{R}^2}u^2dx\right) ^2\\
        \notag&-rt^{2r+2}\ln t\int_{\mathbb{R}^2} u^2dx-\frac{t^{2r+2}}{2}\int_{\mathbb{R}^2}u^2\ln u^2dx.
	\end{align}	
    Observe that $4r+4>2r+2+\kappa>0$, which yields that $I(u_t)\rightarrow-\infty$ as $t\rightarrow+\infty$.
%By $(V_3)$, % we obtain $2r+2+\kappa<4r+4$, therefore, 
%there exists $\tilde t>1$ such that $I(u_{\tilde t})<0.$ This completes the proof. 
\hfill$\square$

\vspace{.2cm}    
Using Lemma \ref{lemma4.2}, we can define the mountain pass value
$$c_{*,2}=\inf_{\gamma\in \Gamma}\max_{t\in [0,1]}I(\gamma (t)),$$
where
$$ \Gamma=\{\gamma\in C([0,1],E)\ |\ \gamma(0)=0, \ I(\gamma(1))<0\}\not=\emptyset.$$
According to Lemma \ref{lemma4.1}, we deduce $c_{*,2}\geq \bar{m}_{\alpha}>0.$

\vspace{.2cm}  
%In the following, recalling \eqref{r1}, we define
Now, we are going to establish a $(PS)_{c_{*,2}}$ sequence for the functional $I$ with an additional property. However, since $\lambda<0$, the auxiliary functional $\Phi$ given by \eqref{a} is no longer applicable. To address this, we 
%However, due to the singularity arising from the logarithmic convolution term and the presence of logarithmic nonlinearity, the boundedness of the Palais–Smale sequence cannot be established. Similarly to the case where \(\lambda>0\), we instead construct  Cerami sequence with an additional property that is crucial to prove its boundedness; to this end, we consider the following Banach space
%	$$\tilde E:=\mathbb{R}\  \times E  $$		
%	with the standard norm $||(b,w)||_{\tilde E}=\left( |b|^2+||w||_E^2\right)^{\frac{1}{2}} $ for $b\in \mathbb{R}, w\in E$. Furthermore, we 
introduce the map $h: \tilde E\rightarrow E$ by
	$$ h(b,w)[x]:=e^{rb}w(e^{-b}x), \quad \mathrm{for }\  b\in \mathbb{R}, w\in E \ \mathrm{and}\  x\in\mathbb{R}^2$$
	and the functional $\Psi:\tilde E\rightarrow \mathbb{R}$ by
	\begin{align}\label{4.6}
		\Psi(b,w):=&I(h(b,w))\nonumber\\
        \notag=&\dfrac{e^{2rb}}{2}\int_{\mathbb{R}^2} |\nabla w|^2dx+\dfrac{e^{2b(r+1)}}{2}\int_{\mathbb{R}^2}V\left (e^bx \right) w^2dx+\dfrac{e^{2b(r+1)}}{2}\int_{\mathbb{R}^2}w^2dx\\
        \notag&+\dfrac{\lambda e^{4b(r+1)}}{4}\int_{\mathbb{R}^2}\int_{\mathbb{R}^2}\ln |x-y|w^2(x)w^2(y)dxdy+\dfrac{\lambda be^{4b(r+1)}}{4}\left( \int_{\mathbb{R}^2} w^2dx\right) ^2\\
        &-rbe^{2b(r+1)}\int_{\mathbb{R}^2} w^2dx-\frac{e^{2b(r+1)}}{2}\int_{\mathbb{R}^2} w^2\ln w^2dx.
	\end{align}	

    Similar to Lemma \ref{lemma2.15}, we have the following lemma.
	\begin{lemma}
		It holds $\Psi\in C^1{(\tilde E,\mathbb{R})}$. Furthermore,
		\begin{align}
		\Psi_b(b,w)\notag&=re^{2rb}\int_{\mathbb{R}^2}|\nabla w|^2dx+(r+1)e^{2b(r+1)}\int_{\mathbb{R}^2}V\left({e^b x} \right) w^2dx\\
        \notag&\quad+\frac{e^{2b(r+1)}}{2}\int_{\mathbb{R}^2}\left( \nabla V\left( {e^bx}\right) \cdot(e^bx) \right) w^2dx\\\notag&\quad +e^{2b(r+1)}\int_{\mathbb{R}^2}w^2dx+\lambda (r+1)e^{4b(r+1)}\int_{\mathbb{R}^2}\int_{\mathbb{R}^2}\ln|x-y|w^2(x)w^2(y)dxdy\\
        \notag&\quad+\left(\lambda b(r+1)e^{4b(r+1)}+\dfrac{\lambda e^{4b(r+1)}}{4} \right)\left( \int_{\mathbb{R}^2}w^2dx\right)^2\\
        \notag&\quad-(r+1)e^{2b(r+1)}\int_{\mathbb{R}^2} w^2\ln w^2dx -2r(r+1)be^{2b(r+1)} \int_{\mathbb{R}^2} w^2dx 
			%\\\notag&=r\int_{\mathbb{R}^2} |\nabla h|^2dx+\int_{\mathbb{R}^2}h^2dx+(r+1)\int_{\mathbb{R}^2} V(x)h^2dx\\\notag&\quad +\dfrac{1}{2}\int_{\mathbb{R}^2} (\nabla V\cdot x)h^2dx+\lambda(r+1)\int_{\mathbb{R}^2}\int_{\mathbb{R}^2} \ln |x-y|h^2(x)h^2(y)dxdy\\\notag&\quad +\dfrac{\lambda}{4}\left(\int_{\mathbb{R}^2} h^2dx\right) ^2-(r+1)\int_{\mathbb{R}^2} h^2\ln h^2dx=rI'(h)h+P(h)=T(h)
		\end{align}		
		and 
		\begin{align}
\langle\Psi_w(b,w),v\rangle\notag&=e^{2rb}\int_{\mathbb{R}^2}\nabla v\nabla wdx+e^{2b(r+1)}\int_{\mathbb{R}^2}V\left({e^bx} \right)vwdx \\\notag&\quad +\lambda  e^{4b(r+1)}\int_{\mathbb{R}^2}\int_{\mathbb{R}^2}\ln |x-y|w^2(y)w(x)v(x)dxdy\\\notag&\quad+\lambda be^{4b(r+1)}\int_{\mathbb{R}^2}w^2(y)dy\int_{\mathbb{R}^2}v(x)w(x)dx\\\notag&\quad-2rbe^{2b(r+1)}\int_{\mathbb{R}^2}vwdx-e^{2b(r+1)}\int_{\mathbb{R}^2}vw\ln w^2dx\\\notag&=\langle I'(h(b,w)),h(b,v)\rangle,
		\end{align}		
		$	\mathrm{for}\ \mathrm{all}\ b\in \mathbb{R}\ \mathrm{and}\ v, w\in E	$.
	\end{lemma}	
%\noindent   {\it Proof.} Since the argument is similar to Lemma \ref{lemma2.15}, we omit the proof.\hfill$\square$

  By an argument analogous to that of Lemma \ref{lemma2.9}, we derive a Cerami sequence of the functional $I$ with an additional property.
	\begin{lemma}\label{lemma4.4}
		There exists $\{u_n\}\subset  E$ such that
		\begin{align*}
			I(u_n)\rightarrow c_{*,2}, \quad ||I'(u_n)||_{E'}(1+||u_n||_E)\rightarrow0\quad \mathrm{and }\quad  T(u_n)\rightarrow0
		\end{align*}
		as $n\rightarrow\infty,$ where 
        \begin{align}
		T(u)%=r\langle I'(u),u\rangle +P(u)
        \notag&:=r\int_{\mathbb{R}^2} |\nabla u|^2dx+\int_{\mathbb{R}^2}u^2dx+(r+1)\int_{\mathbb{R}^2} V(x)u^2dx\\\notag&\quad +\dfrac{1}{2}\int_{\mathbb{R}^2} (\nabla V\cdot x)u^2dx+\lambda(r+1)\int_{\mathbb{R}^2}\int_{\mathbb{R}^2} \ln |x-y|u^2(x)u^2(y)dxdy\\\notag&\quad +\dfrac{\lambda}{4}\left(\int_{\mathbb{R}^2} u^2dx\right) ^2-(r+1)\int_{\mathbb{R}^2} u^2\ln u^2dx.
	\end{align}
	\end{lemma}	
	%\noindent   {\it Proof.} The proof is similar to that of Lemma \ref{lemma2.9}, and is therefore omitted here. \hfill$\square$

    In the following, we will show that the Cerami sequence given in Lemma \ref{lemma4.4} is bounded in $E$.
	
	\begin{lemma}\label{lemma4.5}
		Let $\{u_n\}\subset  E$ be a sequence  such that
		\begin{align}\label{4.9}
			c:=\sup_{n\in \mathbb{N}}I(u_n)<+\infty, \quad ||I'(u_n)||_{E'}(1+||u_n||_E)\rightarrow0\quad \mathrm{and }\quad  T(u_n)\rightarrow0
		\end{align}
		as $n\rightarrow\infty.$ Then, $\{u_n\}$ is bounded in $E$.
	\end{lemma}				
		
	\noindent   {\it Proof.} First, using Lemma \ref{lemmaa} and $(V_2)$, we obtain
	\begin{align}\label{4.11}
		c+o_n(1)\notag&\geq I(u_n)-\dfrac{1}{4(r+1)}T(u_n)\\\notag&=\dfrac{r+2}{4(r+1)}\int_{\mathbb{R}^2}|\nabla u_n|^2dx+\dfrac{2r+1}{4(r+1)}\int_{\mathbb{R}^2}u_n^2dx\\\notag&\quad+\frac{1}{4}\int_{\mathbb{R}^2}V(x)u_n^2dx-\dfrac{1}{8(r+1)}\int_{\mathbb{R}^2}\left( \nabla V\cdot x\right)u_n^2dx\\\notag&\quad-\frac{\lambda}{16(r+1)}\left( \int_{\mathbb{R}^2}u_n^2dx\right)^2-\dfrac{1}{4}  \int_{\mathbb{R}^2}u_n^2\ln u_n^2dx\\
        \notag&\geq\dfrac{r+2}{4(r+1)}\int_{\mathbb{R}^2}|\nabla u_n|^2dx+\dfrac{2r+1}{4(r+1)}\int_{\mathbb{R}^2}u_n^2dxx\\
        &\quad+\frac{2r+2-\kappa}{8(r+1)}\int_{\mathbb{R}^2}V(x)u_n^2d-\frac{\lambda}{16(r+1)}\left( \int_{\mathbb{R}^2}u_n^2dx\right)^2-C  \int_{\mathbb{R}^2}|u_n|^qdx,
	\end{align}
	where $2<q<3$.
	
    Next, we claim
	\begin{align}\label{4.12}
	\int_{\mathbb{R}^2}|\nabla u_n|^2dx\leq C, \quad \mathrm{for} \ \ n\in \mathbb{N}.
	\end{align} 
	Suppose, for contradiction, that the claim fails. Upon passing to a subsequence if necessary, we may assume that 	
	\begin{align}
	\notag 	\int_{\mathbb{R}^2}|\nabla u_n|^2dx\rightarrow \infty, \quad \mathrm{as} \ \ n\rightarrow \infty. 
		\end{align}
		We set $s_n:=||\nabla u_n||^{{-\frac{1}{2}}}_2$ for $n\in \mathbb{N}$, which means $s_n\rightarrow 0$ as $n\rightarrow\infty$. Define $w_n(x):=s_n^2 u_n(s_n x) \in E$. %We then introduce the rescaled functions  $w_n\in E$ via $w_n(x)=s_n^2 u_n(s_n x) $. It follows that
		%\begin{align}\label{4.13}
		%\int_{\mathbb{R}^2}|\nabla w_n|^2dx=s_n^4\int_{\mathbb{R}^2} |\nabla u_n|^2dx=1
		%\end{align}
		%and
		%\begin{align}\label{4.14}
%\int_{\mathbb{R}^2}|w_n|^qdx=s_n^{2q}\int_{\mathbb{R}^2}\left[ u_n(s_nx)\right] ^qdx=s_n^{2q-2}\int_{\mathbb{R}^2}|u_n|^qdx.
		%\end{align}
		%Using the Gagliardo-Nirenberg, we deduce
	%	\begin{align}
	%	\notag	\int_{\mathbb{R}^2}|w_n|^qdx\leq C \int_{\mathbb{R}^2}w_n^2dx\left( \int_{\mathbb{R}^2}|\nabla w_n|^2dx\right) ^{\frac{q-2}{2}}=C\int_{\mathbb{R}^2}w_n^2dx.
	%	\end{align}
		Multiplying both sides of \eqref{4.11} by  $s_n^4$ and invoking \eqref{2-28}-\eqref{3.35}, we derive
		\begin{align}\label{l2}
		c	{s_n^4}+o\left( s_n^4\right) 
        \notag&\geq\dfrac{s_n^4(r+2)}{4(r+1)}\int_{\mathbb{R}^2}|\nabla u_n|^2dx+\dfrac{s_n^4(2r+1)}{4(r+1)}\int_{\mathbb{R}^2}u_n^2dx-Cs_n^4 \int_{\mathbb{R}^2}|u_n|^qdx\\
        \notag&\quad+\frac{s_n^4\left(2r+2-\kappa \right)}{8(r+1)}\int_{\mathbb{R}^2}V(x)u_n^2dx-\frac{\lambda s_n^4}{16(r+1)}\left( \int_{\mathbb{R}^2}u_n^2dx\right)^2\\
        \notag&=\dfrac{r+2}{4(r+1)}+\dfrac{s_n^2(2r+1)}{4(r+1)}\int_{\mathbb{R}^2}w_n^2dx+\frac{s_n^2\left(2r+2-\kappa \right)}{8(r+1)}\int_{\mathbb{R}^2}V\left({{s_nx}} \right) w_n^2dx\\
        \notag&\quad-\frac{\lambda}{16(r+1)}\left( \int_{\mathbb{R}^2}w_n^2dx\right)^2 -Cs_n^{6-2q}\int_{\mathbb{R}^2}|w_n|^qdx\\
        \notag&\geq\dfrac{r+2}{4(r+1)}+\dfrac{s_n^2(2r+1)}{4(r+1)}\int_{\mathbb{R}^2}w_n^2dx+\frac{s_n^2\left(2r+2-\kappa \right)}{8(r+1)}\int_{\mathbb{R}^2}V\left({{s_nx}} \right) w_n^2dx\\
        \notag&\quad-\frac{\lambda}{16(r+1)}\left( \int_{\mathbb{R}^2}w_n^2dx\right)^2 -Cs_n^{6-2q}\int_{\mathbb{R}^2}|w_n|^2dx\\
        &\geq%\dfrac{s_n^2(2r+1)}{4(r+1)}\int_{\mathbb{R}^2}w_n^2dx
        -\frac{\lambda}{16(r+1)}\left( \int_{\mathbb{R}^2}w_n^2dx\right)^2 -Cs_n^{6-2q}\int_{\mathbb{R}^2}|w_n|^2dx,
		\end{align}
		which implies 
\begin{align}
		\int_{\mathbb{R}^2} w_n^2dx&=O\left(s_n^{6-2q} \right),\label{4.15}\\
        s_n^2\int_{\mathbb{R}^2} V\left( s_n x\right)w_n^2dx
    \leq &Cs_n^4+{Cs_n^{6-2q}}\int_{\mathbb{R}^2} w_n^2dx
    =O\left(s_n^{12-4q} \right).\label{l1}
	\end{align}
Using \eqref{l2}-\eqref{l1}, we get
\begin{align*}
   C s_n^4\geq \dfrac{r+2}{4(r+1)}+o_n(1),
\end{align*}
which is a contradiction. Hence, \eqref{4.12}
holds.

It follows from the Gagliardo-Nirenberg inequality and \eqref{4.12} that
$$\int_{\mathbb{R}^2}|u_n|^qdx\leq C \int_{\mathbb{R}^2}u_n^2dx\left( \int_{\mathbb{R}^2}|\nabla u_n|^2dx\right) ^{\frac{q-2}{2}}\leq C\int_{\mathbb{R}^2}u_n^2dx.$$ 
Combining this with \eqref{4.11}, we obtain
\begin{align}
	c+o_n(1)%\notag&\geq\dfrac{r+2}{4(r+1)}\int_{\mathbb{R}^2}|\nabla u_n|^2dx+\dfrac{2r+1}{4(r+1)}\int_{\mathbb{R}^2}u_n^2dx+\frac{1}{4}\int_{\mathbb{R}^2}V(x)u_n^2dx\\
  %  \notag&\quad-\dfrac{1}{8(r+1)}\int_{\mathbb{R}^2}\left( \nabla V\cdot x\right)u_n^2dx-\frac{\lambda}{16(r+1)}\left( \int_{\mathbb{R}^2}u_n^2dx\right)^2-C  \int_{\mathbb{R}^2}|u_n|^qdx\\
  %  \notag&\geq\dfrac{r+2}{4(r+1)}\int_{\mathbb{R}^2}|\nabla u_n|^2dx+\dfrac{2r+1}{4(r+1)}\int_{\mathbb{R}^2}u_n^2dx+\frac{2(r+1)-\kappa}{8}\int_{\mathbb{R}^2}V(x)u_n^2dx\\
   % \notag&\quad-\frac{\lambda}{16(r+1)}\left( \int_{\mathbb{R}^2}u_n^2dx\right)^2-C  \int_{\mathbb{R}^2}u_n^{\frac{5}{2}}dx\\
    \notag&\geq\dfrac{r+2}{4(r+1)}\int_{\mathbb{R}^2}|\nabla u_n|^2dx+\dfrac{2r+1}{4(r+1)}\int_{\mathbb{R}^2}u_n^2dx+\frac{2(r+1)-\kappa}{8}\int_{\mathbb{R}^2}V(x)u_n^2dx\\
    \notag&\quad-\frac{\lambda}{16(r+1)}\left( \int_{\mathbb{R}^2}u_n^2dx\right)^2-C  \int_{\mathbb{R}^2}u_n^2dx,
\end{align}		
which means 
$$\int_{\mathbb{R}^2}u_n^2dx\leq C\quad \mathrm{and} \quad \int_{\mathbb{R}^2} V(x)u_n^2dx\leq C.$$
Then,	$\{u_n\}$ is bounded in $E$. This completes the proof. \hfill$\square$

\vspace{.2cm}
The next lemma gives the compactness for the sequence satisfying \eqref{4.9}.

	\begin{lemma}
	Let $\{u_n\}\subset E$ be a sequence satisfying \eqref{4.9}, and assume $\liminf_{n\to \infty}||u_n||_E>0$. Then, %one of the following occurs:
		%\\(i) $||u_n||_E\rightarrow 0$ and $I(u_n)\rightarrow 0$ as $n\rightarrow\infty$.
		there exists $u\in E\backslash\{0\}$ such that $u_n\rightarrow u$ in $E$ as $n\rightarrow\infty$,
and $u$ is a solution of equation \eqref{1-4}.
\end{lemma}	
\noindent   {\it Proof.}	By Lemma \ref{lemma4.5},   $\{u_n\}$  is bounded in $E$, then there exists $u\in E$ such that $u_n\rightharpoonup  u\in E $. Using Lemma \ref{lemma1}, we deduce $u_n\rightarrow u$ in $L^s(\mathbb{R}^2)(2\leq s<\infty).$

In the following, we will prove $u\not \equiv 0$.	
Assume by contradiction $u\equiv 0$, then $u_n\rightarrow 0 $ in $L^s{(\mathbb{R}^2)}$
$	\left( s \geq 2\right). $ Using \eqref{2.8.0} and  \eqref{4.9}, we deduce
\begin{align}
	\notag&\int_{\mathbb{R}^2}\left[ |\nabla u_n|^2+V(x)u_n^2\right] dx-\lambda\int_{\mathbb{R}^2}\int_{\mathbb{R}^2} \ln \left( 1+\dfrac{1}{|x-y|}\right) u_n^2(x)u_n^2(y)dxdy\\
    \notag=&\langle I'(u_n),u_n\rangle -\lambda\int_{\mathbb{R}^2}\int_{\mathbb{R}^2}\ln \left(1+{|x-y|} \right)u_n^2(x)u_n^2(y)dxdy+\int_{\mathbb{R}^2}u_n^2\ln u_n^2dx\\
    \notag\leq &o_n(1) +C||u_n||_2^2\int_{\mathbb{R}^2} V(x)u_n^2dx+C||u_n||_q^q\rightarrow0,\quad \mathrm{as}\quad n\rightarrow\infty,
\end{align}
which implies $||u_n||_E\rightarrow 0$ as $n\rightarrow \infty$. This contradicts $\liminf_{n\to \infty}||u_n||_E>0$. Thus, $u\not \equiv 0$.
%\begin{align}\label{4.18}
%	\int_{\mathbb{R}^2}\left[ |\nabla u_n|^2+V(x)u_n^2\right] dx\rightarrow 0\quad \mathrm{as}\quad n\rightarrow\infty
%\end{align}
%and 
%\begin{align}\label{4.19}
%	\int_{\mathbb{R}^2}\int_{\mathbb{R}^2} \ln \left( 1+\dfrac{1}{|x-y|}\right) u_n^2(x)u_n^2(y)dxdy\rightarrow 0 \quad \mathrm{as}\quad n\rightarrow\infty.
%\end{align} 
%Moreover, it follows from \eqref{w1} that
%\begin{align}\label{w2}
%	\int_{\mathbb{R}^2}\int_{\mathbb{R}^2}\ln (1+|x-y|)u_n^2(x)u_n^2(y)dxdy\leq C||u_n||_2^2\int_{\mathbb{R}^2}V(x)u_n^2dx\rightarrow0, \quad \mathrm{as} \ n\rightarrow\infty.
%\end{align}  
%Then,  using \eqref{4.18}-\eqref{w2}, we obtain
%\begin{align}
%	\int_{\mathbb{R}^2} u_n^2\ln u_n^2dx\notag&=-\langle I'(u_n),u_n\rangle +\int_{\mathbb{R}^2}\left[ |\nabla u_n|^2+V(x)u_n^2\right] dx\\\notag&\quad +\lambda\int_{\mathbb{R}^2}\int_{\mathbb{R}^2}\ln |x-y|u_n^2(x)u_n^2(y)dxdy\rightarrow0\quad \mathrm{as}\quad n\rightarrow\infty.
%\end{align}
%Thus
%\begin{align}
%	I(u_n)\notag&=\dfrac{1}{2}\int_{\mathbb{R}^2}\left[ |\nabla u_n|^2+u_n^2+V(x)u_n^2\right]dx -\frac{1}{2}\int_{\mathbb{R}^2}u_n^2\ln u_n^2dx\\\notag&\quad +\frac{\lambda}{4}\int_{\mathbb{R}^2}\int_{\mathbb{R}^2} \ln |x-y|u_n^2(x)u_n^2(y)dxdy\rightarrow0,\quad \mathrm{as}\quad n\rightarrow\infty,
%\end{align}
%which means a contradiction.

Next,  we prove $u_n\rightarrow u$ in $E$ as $n\rightarrow \infty $. By \eqref{4.9} and \eqref{e1}, we deduce
\begin{align}\label{4.20}
	o_n(1)\notag&=\langle I'(u_n), u_n-u\rangle\\
    \notag&=\int_{\mathbb{R}^2}\left[  \nabla u_n\nabla(u_n-u)+ V(x)u_n(u_n-u)\right] dx-\int_{\mathbb{R}^2} u_n(u_n-u)\ln u_n^2dx\\
    \notag&\quad+\lambda\int_{\mathbb{R}^2}\int_{\mathbb{R}^2} \ln|x-y|u_n^2(y)u_n(x)\left[ u_n(x)-u(x)\right] dxdy\\
   % \notag&=||u_n||^2_E-||u||^2_E+\lambda\int_{\mathbb{R}^2}\int_{\mathbb{R}^2} \ln |x-y|u_n^2(y)u_n(x)(u_n(x)-u(x))dxdy+o_n(1)\\
    \notag&=||u_n||^2_E-||u||^2_E+\lambda\int_{\mathbb{R}^2}\int_{\mathbb{R}^2}\ln\left(1+|x-y| \right)u_n^2(y)u_n(x)\left[ u_n(x)-u(x)\right] dxdy\\
    &\quad-\lambda\int_{\mathbb{R}^2}\int_{\mathbb{R}^2}\ln\left( 1+\dfrac{1}{|x-y|}\right)u_n^2(y)u_n(x)\left[ u_n(x)-u(x)\right] dxdy+o_n(1).
\end{align}
By \eqref{1-7}, we obtain
\begin{align}\label{4.13.1}
\left|\int_{\mathbb{R}^2}\int_{\mathbb{R}^2}\ln\left( 1+\dfrac{1}{|x-y|}\right)u_n^2(y)u_n(x)(u_n(x)-u(x))dxdy\right|
\leq C||u_n||_{\frac{8}{3}}^3||u_n-u||_{\frac{8}{3}}\rightarrow 0.
\end{align}
We deduce from \eqref{w1} and the H\"older inequality that 
\begin{align}\label{l3}
\notag&	\quad\left|\int_{\mathbb{R}^2}\int_{\mathbb{R}^2}\ln (1+|x-y|)u_n^2(y)u_n(x)\left[ u_n(x)-u(x)\right] dxdy\right|\\
\notag&\leq\int_{\mathbb{R}^2}\int_{\mathbb{R}^2}\ln (1+|x|)u_n^2(y)\left|u_n(x)\left[ u_n(x)-u(x)\right] \right|dxdy\\
\notag&\quad+\int_{\mathbb{R}^2}\int_{\mathbb{R}^2}\ln (1+|y|)u_n^2(y)\left|u_n(x)\left[ u_n(x)-u(x)\right] \right|dxdy\\
\notag&\leq||u_n||_2^2\left( \int_{\mathbb{R}^2}\ln ^2(1+|x|)u_n^2dx\right)^{\frac{1}{2}}||u_n-u||_2+C||u_n||_2||u_n-u||_2\int_{\mathbb{R}^2}V(y)u_n^2(y)dy\\
&\leq C\left( \int_{\mathbb{R}^2}\ln ^2(1+|x|)u_n^2dx\right)^{\frac{1}{2}}||u_n-u||_2+o_n(1).
\end{align}
Analogous to \eqref{2.1.1}, we have $\ln ^2(1+|x|)\leq CV(x)$. Hence, 
\begin{align*}
    \int_{\mathbb{R}^2}\ln ^2(1+|x|)u_n^2dx
    \leq C\int_{\mathbb{R}^2} V(x)u_n^2dx\leq C,
\end{align*}
which, together with \eqref{l3}, implies
\begin{align}\label{4.15.1}
    \int_{\mathbb{R}^2}\int_{\mathbb{R}^2}\ln (1+|x-y|)u_n^2(y)u_n(x)\left[ u_n(x)-u(x)\right] dxdy
    =o_n(1).
\end{align}
%\begin{align}
%\left|	\int_{\mathbb{R}^2}\ln (1+|x|)u_n(x)\left[ u_n(x)-u(x)\right] dx\right|
%\notag&\leq\left( \int_{\mathbb{R}^2}\ln ^2(1+|x|)u_n^2dx\right)^{\frac{1}{2}}||u_n-u||_2\\
%\notag&\leq C||u_n-u||_2\left(\int_{\mathbb{R}^2}|x|^\gamma u_n^2(x)dx \right)^{\frac{1}{2}} \\
%\notag&\leq C||u_n-u||\left(\int_{\mathbb{R}^2} V(x)u_n^2dx \right)^{\frac{1}{2}} \rightarrow0, \quad \mathrm{as}\quad n\rightarrow\infty.
%\end{align}
Thus, it follows from \eqref{4.20}, \eqref{4.13.1} and \eqref{4.15.1} that $||u_n||_E\rightarrow ||u||_E$ as $n\rightarrow\infty$. Combining this with $u_n\rightharpoonup u$ in $E$, we obtain $u_n\rightarrow u$ in $E$.
%Finally,  we need to prove 	$I'(u)=0$. Let $v\in E$, it follows from \eqref{4.9} that
%\begin{align}
%	\notag	|I'(u)v|=\lim_{n\to \infty}|I'(u_n)v|\leq \lim_{n\to \infty}\left[||I'(u_n)||_E||v||_{E'} \right]=0,
%\end{align}
Thus, $u$ is a nontrivial solution of equation \eqref{1-4}.	\hfill$\square$

\vspace{.2cm}
\noindent{\textbf {Proof of Theorem \ref{theorem1.3}.}}
 The proof is similar to that of Theorem \ref{theorem1.1}, and is therefore omitted here.

	\section{The symmetric setting}\label{s5}

This section focuses on the proof of Theorem \ref{theorem1.2} and Corollary \ref{corollary1}. Some arguments in the proof of Theorem \ref{theorem1.2} are analogous to those of Theorem \ref{theorem1.1} and will therefore only be outlined. 

\vspace{.2cm}
\noindent\textbf {Proof of Theorem \ref{theorem1.2}.} We prove only the case $\lambda> 0$, as the proof for $\lambda< 0$ is similar.
It follows from Lemmas
\ref{lemma2.6} and \ref{lemma2.7} that the functional $I$ possesses a mountain pass geometry on $E_G$. 
%Thus, the definition of the mountain pass value given below
Moreover, 
$$0<m_{\alpha}\leq c_{G}=\inf_{\gamma\in \Gamma_G}\max_{s\in [0,1]}I(\gamma(s))<\infty,$$
 where
$$\Gamma_G:=\left\{\gamma\in C\left([0,1], E_G\right)\ |\  \gamma(0)=0, I(\gamma(1))<0 \right\}.$$
 Furthermore, proceeding analogously to the proof of Lemma \ref{lemma2.9}, there exists $\{u_n\}\subset E_G$ such that 
\begin{align*}%\label{q3}
	I(u_n) \to c_{G}, \quad \|I'(u_n)\|_{E_G'} \bigl(1 + \|u_n\|_{E_G}\bigr) \to 0 \quad \text{and} \quad K(u_n) \to 0,\quad \mathrm{as}\quad n\rightarrow\infty.
\end{align*}
%where $K$ is defined \eqref{q3}.  
Similar to Lemma \ref{lemma2.11}, there exists $u\in E_G$ such that $u_n\to u$ in $E$ and moreover $\|I'(u)\|_{E_G'}=0$.
%For any  $u_n\in E_G$ satisfying \eqref{q1}, it follows from  \eqref{q1} that
%\begin{align}
%	\notag 	[f\diamond u](x)=\lim_{n\rightarrow\infty}[f\diamond u_n](x)=\varrho (f)\lim_{n\to \infty} u_n(f^{-1}x)=\lim_{n\rightarrow\infty} u_n=u,
%\end{align}
% which means $u\in E_G.$  Furthermore,  similar to Lemma \ref{lemma2.10} and \ref{lemma2.11}, we deduce $$ I'(u)v = 0\ \mathrm{for} \ \mathrm{all}\ u \in E_G\ and \ v \in E_G^\perp,$$ where \(E_G^\perp\) denotes the orthogonal complement of \(E_G\) in \(E\). 
Using the principle of symmetric criticality (see \cite[Theorem 1.28]{W}), we obtain 
  $$\|I'(u)\|_{E'}=0.$$
Then, $u$ is a G-invariant solution of equation \eqref{1-4} with $I(u) =c_{G}$.  Thus, we can define
$$c_{*,G}:=\inf\{I(w)\ |\ w\in E_G\backslash\{0\} \ \mathrm{solves}\ \eqref{1-4} \}.$$ 
By a proof similar to that of Theorem \ref{theorem1.1}, we deduce that $c_{*,G}$ is achieved by some $\tilde u\in E_G\setminus\{0\}$ and $\tilde u$ is a $G$-invariant ground state solution of equation \eqref{1-4}. \hfill$\square$

\vspace{.2cm}
\noindent\textbf {Proof of Corollary \ref{corollary1}.} Similarly to Example 1.1, suppose that for any given $n\in\mathbb{N}_+$, the subgroup $G_n\subset O(2)$ of order $2\cdot 3^n$ is generated by the $\frac{\pi}{3^n}$-rotation (counter-clockwise)
$$f_n\in O(2),\quad   f_n:=\begin{pmatrix}
\cos \frac{\pi}{3^n} & -\sin \frac{\pi}{3^n} \\
 \sin \frac{\pi}{3^n} & \cos \frac{\pi}{3^n}
\end{pmatrix},$$
then 
$$f_n x=\left(x_1\cos \frac{\pi}{3^n}-x_2 \sin \frac{\pi}{3^n}, x_1 \sin \frac{\pi}{3^n}+x_2\cos \frac{\pi}{3^n}\right), \quad \mathrm{for}\ x=(x_1,x_2)\in \mathbb{R}^2.$$
Let $\sigma_n\colon G_n\to\{-1,1\}$ denote the homomorphism defined by
\[
\sigma_n(f_n^i)=(-1)^i \quad\text{for }\ i=1,2,\cdots,2\cdot 3^n,
\]
where $f_n^i$ is the $\frac{i\pi}{3^n}$-rotation.
Clearly, $E_{G_{n+1}}\subset E_{G_n}$ for all $n\in\mathbb{N}_+$, which yields
$$c_{G_{n+1}}\geq c_{G_n}\geq c_{G_1}, \mathrm{for}\ \mathrm{all}\ n\in \mathbb{N}_+.$$
From Theorem \ref{theorem1.2}, we conclude that equation \eqref{1-4} admits a nonradial sign-changing solution $u_n\in{E}_{G_n}$ satisfying $I(u_n)=c_{G_n}$ for every $n\in\mathbb{N}_+$.

Next, we will prove
$$I(u_n)=c_{G_n}\rightarrow +\infty, \quad \mathrm{as}\ n\rightarrow\infty. $$
Suppose, by contradiction, that there exists a constant $C>0$ such that 
$$\limsup_{n\rightarrow\infty}c_{G_n}\leq C \quad \mathrm{for}\ \mathrm{all} \ n\in\mathbb{N}_+.$$ 
Then, arguing as in Lemma \ref{lemma2.10}, we deduce that $\{u_n\}$ is bounded in $E$. In view of Lemma \ref{lemma2.11}, we obtain that $$u_n\rightarrow \tilde u \ \ \mathrm{in}\ \  E \quad \mathrm{as}\ n\rightarrow\infty,$$ 
%Due to  $Fix(G)=\{0\}$, $\tilde u\in E\setminus \{0\}$ holds. 
%Since $\|I'(u_n)\|_{{E}^{-1}} \to 0$, passing to the limit yields $\|I'(\tilde u)\|_{{E}^{-1}} = 0$, which implies that 
and $\tilde u$ is a nontrivial solution of equation \eqref{1-4}. By virtue of Lemma \ref{lemma2.13}, we have $\tilde u \in C(\mathbb{R}^2)$. For any fixed $m \in \mathbb{N}_+$, we have
\[
u_n \in E_{G_n} \subset E_{G_m}\quad \mathrm{for\ all} \ n \geq m.
\]
Then, we derive
\[
\tilde u \in E_{G_m} \quad \text{for every } m \in \mathbb{N}_+,
\]
which yields
\[
\tilde u(x) = -\tilde u(f_m x), \quad \mathrm{for\ all}\, x \in \mathbb{R}^2,\; m \in \mathbb{N}_+.
\]
It is known that $f_mx\rightarrow x$ for $x\in \mathbb{R}^2$ as $n\rightarrow\infty$, thus $\tilde u=0$, which means a contradiction.  This completes the proof.  \hfill$\square$

\bigskip

\noindent{\bf{Data availability statement}}  

No new data was created or analysed in this study.
\smallskip

\noindent{\bf{Acknowledgments}} 

This work is supported by National Natural Science Foundation of China (Nos. 12671136, 12271223), National Natural Science Foundation of Jiangxi
(Nos. 20252BAC230002, 20242BAB26001) and Double-high talents in Jiangxi Province.
\smallskip

\noindent{\bf{Conflict of interest statement}}  

All authors declare that they have no conflict of interest.


\smallskip



		
		\begin{thebibliography}{99}
			{\footnotesize


      %\bibitem{AC} M. Alfaro, R. Carles, {\it Superexponential growth or decay in the heat equation  with a logarithmic nonlinearity}, Dyn. Partial Differ. Equ., {\bf 14} ( 2017) 343-358.
      
        \bibitem{A}  A. Ambrosetti,  {\it On Schr\"{o}dinger-Poisson systems}, Milan J. Math., {\bf 76} (2008) 257–274. 

        \bibitem{Az}
       A. Azzollini,
       {\it The planar Schr\"{o}dinger-Poisson system with a positive potential}, 
      Nonlinearity, {\bf 34} (2021) 5799–5820.
        
      % \bibitem{ADP}  A. Azzollini, P. d’Avenia, A. Pomponio,  {\it Schr\"{o}dinger equation in dimension two with competing logarithmic self-interaction}, Calc. Var. Partial Differential Equations,   {\bf 64} (2025) 130-152.
 \bibitem{BJL} J. Bellazzini, L. Jeanjean, T. Luo,  {\it Existence and instability of standing waves with prescribed norm for a class of Schr\"{o}dinger-Poisson equations}, Proc. London Math. Soc.,  {\bf 107} (2013) 303–339.
      
       \bibitem{BF}
V. Benci, D. Fortunato,
{\it An eigenvalue problem for the Schr\"{o}dinger-Maxwell equations},
Topol. Methods Nonlinear Anal.,
{\bf 11} (1998) 283--293.

\bibitem{BBL}
R. Benguria, H. Br\'ezis, E. Lieb,
{\it The Thomas-Fermi-von Weizs\"acker theory of atoms and molecules},
Comm. Math. Phys., {\bf 79} (1981) 167--180.


      % \bibitem{BM1}  I. Białynicki-Birula,  J. Mycielski,  {\it Nonlinear wave mechanics}, Ann. Physics, {\bf 100} (1976) 62–93.
       
     %  \bibitem{BM2} I. Białynicki-Birula,   J. Mycielski,  {\it Wave equations  with logarithmic nonlinearities}, Bull. Acad. Polon. Sci. Sér. Sci. Math. Astronom. Phys., {\bf 23} (1975) 461–466.
       
        
       \bibitem{CG} R. Carles, I. Gallagher,   {\it Universal dynamics for the defocusing logarithmic Schr\"{o}dinger equation}, Duke Math. J.,  {\bf 167} (2018) 1761–1801.

       \bibitem{CL}
I. Catto, P.-L. Lions,
{\it Binding of atoms and stability of molecules in Hartree and Thomas-Fermi type theories},
Comm. Partial Differential Equations, {\bf 18} (1993) 1149--1159.

\bibitem{CV}
G. Cerami, G. Vaira, 
{\it Positive solutions for some non-autonomous Schr\"{o}dinger-Poisson systems},
J. Differential Equations,
{\bf 248} (2010) 521--543.

\bibitem{CRT}
S. Chen, V. Rădulescu, X. Tang, 
{\it Multiple normalized solutions for the planar Schr\"{o}dinger-Poisson system with critical exponential growth},
Math. Z.,
{\bf 306} (2024) Paper No. 50, 32 pp.

%\bibitem{CJ} 
%S. Cingolani, L. Jeanjean,
%{\it Stationary waves with prescribed $L^2$-norm for the planar Schr\"{o}dinger-Poisson system},
% SIAM J. Math. Anal.,
% {\bf 51} (2019) 3533–3568.
       
       \bibitem{CW} S. Cingolani,  T. Weth,{ \it On the planar Schr\"{o}dinger-Poisson system}, 
       Ann. Inst. H. Poincaré C Anal. Non Linéaire, {\bf 33} (2016) 169–197. 
       
       %	\bibitem{1}   S. Cingolani, T. Weth,  On the planar Schr\"{o}dinger-Poisson system, Ann. Inst. Henri Poincare, 33 (2016) 169–97.
      \bibitem{DMS}   P. d’Avenia, E. Montefusco,  M. Squassina,  {\it On the logarithmic Schr\"{o}dinger equation}, Commun. Contemp. Math.,  {\bf 16} (2014) Paper No. 1350032, 15 pp.

\bibitem{DFJ}
     J. Dolbeault, R. Frank, L. Jeanjean, 
{\it Logarithmic estimates for mean-field models in dimension two and the Schr\"{o}dinger-Poisson system},
C. R. Math. Acad. Sci. Paris, {\bf 359} (2021) 1279–1293.


       \bibitem{DW} M. Du, T. Weth, {\it Ground states and high energy solutions of the planar Schr\"{o}dinger-Poisson system}, Nonlinearity,  {\bf 30} (2017) 3492–3515. 
       
        \bibitem{FW}
H. Fan, Q. Wang, 
{\it Prescribed number nodes of nodal solutions for a planar Schr\"{o}dinger-Poisson system},
 J. Differential Equations,
{\bf 475} (2026), Paper No. 114457, 52 pp.


        \bibitem{GPV}  V. Georgiev, F. Prinari, N. Visciglia, {\it On the radiality of constrained minimizers to the Schr\"{o}dinger-Poisson-Slater energy}, Ann. Inst. H. Poincaré C Anal. Non Linéaire, {\bf 29 }(2012) 369–376.

\bibitem{GLL}
        Y. Guo, W. Liang, Y. Li, 
       {\it Existence and uniqueness of constraint minimizers for the planar Schr\"{o}dinger-Poisson system with logarithmic potentials}, 
       J. Differential Equations, 
       {\bf 369} (2023) 299–352.

     % \bibitem{HIT} J. Hirata, N. Ikoma, K. Tanaka,  {\it  Nonlinear scalar field equations  in $R^N$: mountain pass, symmetric mountain pass approaches}, Topol. Methods Nonlinear Anal., {\bf 35} (2010) 253–276.

\bibitem{JWZ} Y. Jiang, Z. Wang, H. Zhou, {\it Positive solutions for Schr\"{o}dinger-Poisson-Slater system with coercive potential}, Topol. Methods Nonlinear Anal., {\bf 57} (2021) 427--439.

        \bibitem{JZ} Y. Jiang, H. Zhou,  {\it Schr\"{o}dinger-Poisson system with steep potential well}, J. Differential Equations, {\bf 251} (2011) 582–608.
        
        \bibitem{J} L. Jeanjean,   {\it Existence of solutions with prescribed norm for semilinear elliptic equations},  Nonlinear Anal.,  {\bf 28} (1997) 1633–1659.

        \bibitem{LLTY}
  B. Li, W. Long, Z. Tang, J. Yang, 
  {\it Uniqueness of positive bound states with multiple bumps for Schr\"{o}dinger-Poisson system}, 
  Calc. Var. Partial Differential Equations, 
  {\bf 60} (2021) Paper No. 240, 28 pp.
					
		\bibitem{3} G. Li,  C. Wang,  {\it The existence of a nontrivial solution to a nonlinear elliptic problem of linking type without the Ambrosetti-Rabinowitz condition}, Ann. Acad. Sci. Fenn. Math., {\bf 36} (2011) 461–480.

        % \bibitem{19} G. Li, S. Peng, S. Yan,  {\it Infinitely many positive solutions for the nonlinear Schr\"{o}dinger-Poisson system}, Commun. Contemp. Math., {\bf 12} (2010) 1069–92.

     \bibitem{25} E. Lieb, {\it Sharp constants in the Hardy-Littlewood-Sobolev and related inequalities}, Ann. of Math. (2), {\bf 118} (1983) 349–374.

     \bibitem{L}
E. Lieb,
{\it Thomas-Fermi and related theories of atoms and molecules},
Rev. Modern Phys., {\bf 53} (1981) 603--641.

\bibitem{Lions}
P.-L. Lions, {\it Solutions of Hartree-Fock equations for Coulomb systems},
Comm. Math. Phys., {\bf 109} (1987) 33--97.
         
	   %\bibitem{LM2} S. Liu, S. Mosconi, {\it On the Schr\"{o}dinger-Poisson system with indefinite
			%		potential and 3-sublinear nonlinearity}, J. Differential Equations, {\bf 269} (2020) 689-712.

    

    
    
      \bibitem{LM1} S. Liu, S. Mosconi,  {\it On the Schr\"{o}dinger-Poisson system with indefinite potential and 3-sublinear nonlinearity},  J. Differential Equations,  {\bf 269} (2020) 689–712.

      \bibitem{LRTZ}
      Z. Liu, V. Rădulescu, C. Tang, J. Zhang, 
      {\it Another look at planar Schr\"{o}dinger-Newton systems}, J. Differential Equations,
     {\bf 328} (2022) 65–104.

\bibitem{LRZ}
      Z. Liu, V. Rădulescu, J. Zhang, 
     {\it A planar Schr\"{o}dinger-Newton system with Trudinger-Moser critical growth}, 
     Calc. Var. Partial Differential Equations,
    {\bf 62} (2023) Paper No. 122, 31 pp.

      \bibitem{MRS}
P. Markowich, C. Ringhofer, C. Schmeiser,
{\it Semiconductor Equations}, Springer-Verlag, Vienna, (1990).

 \bibitem{R} D. Ruiz, {\it The Schr\"{o}dinger-Poissom equation under the effect of a nonlinear local term}, J. Funct. Anal., {\bf 237} (2006) 655--674. 

\bibitem{SSY1}
L. Shan, W. Shuai, J. Ye,
{\it Multiple solutions for the nonlinear Schr\"odinger-Poisson system with a partial confinement},
J. Differential Equations, {\bf 451} (2026) Paper No. 113815, 41 pp.

 \bibitem{S}   W. Shuai, {\it Multiple solutions for logarithmic Schr\"{o}dinger equations}, Nonlinearity, {\bf 32} (2019) 2201–2225.

  %    \bibitem{SS1}  M. Squassina, A. Szulkin,  {\it Multiple solutions to logarithmic Schr\"{o}dinger equations with periodic potential}, Calc. Var. Partial Differential Equations,  {\bf 54} (2015) 585–597.
				   
  %    \bibitem{SS2} M. Squassina, A. Szulkin,  {\it Erratum to: multiple solutions to logarithmic Schr\"{o}dinger equations with periodic potential},  Calc. Var. Partial Differential Equations, {\bf 56} (2017) Paper No. 56, 4 pp.

      \bibitem{S1} J. Stubbe, {\it Bound states of two-dimensional Schr\"{o}dinger-Newton equations}, arXiv:0807.4059, (2008).

\bibitem{T}
W. Troy,
    {\it Uniqueness of positive ground state solutions of the logarithmic Schr\"{o}dinger equation},
     Arch. Ration. Mech. Anal., 
    {\bf 222} (2016) 1581–1600.
				
	 % \bibitem{13}	  W. Shuai, Multiple solutions for logarithmic Schr\"{o}dinger equations, Nonlinearity 32 (2019) 2201–2225. 
     \bibitem{WZZ} Z. Wang, X. Zeng, H. Zhou, {\it Nonradial symmetry and blow-up of ground states for a Schr\"{o}dinger-Poisson system with logarithmic term}, Nonlinearity,  {\bf 38} (2025) Paper No. 015011, 23 pp.

   
     \bibitem{WZ} Z. Wang, H. Zhou,  {\it Sign-changing solutions for the nonlinear Schr\"{o}dinger-Poisson system in $\mathbb{R}^3$}, Calc. Var. Partial Differential Equations,  {\bf 52} (2015) 927–943.   

    
\bibitem{WZ1}
    Z. Wang, C. Zhang,
    {\it Convergence from power-law to logarithm-law in nonlinear scalar field equations},
     Arch. Ration. Mech. Anal., 
     {\bf 231} (2019) 45–61.
    
    \bibitem{W} M. Willem,  {\it Minimax Theorems}, Birkh\"{a}user Boston (1996).
      
    
    

     
     
     
 
 






				
			}
		\end{thebibliography}
	\end{document}